\documentclass[11pt, reqno]{amsart}
\usepackage{amsfonts}
\usepackage{amssymb}
\usepackage{latexsym}
\usepackage[]{enumerate}
\usepackage{verbatim}
\usepackage{xcolor}

\newcommand{\norm}[1]{\left\Vert #1 \right\Vert}

\newcommand{\R}{\mathbb{R}}

\DeclareMathOperator{\sgn}{sgn}

\newtheorem{theorem}{Theorem}

\newtheorem{lemma}{Lemma}
\newtheorem{corollary}{Corollary}

\theoremstyle{definition}
\newtheorem{definition}{Definition}

\theoremstyle{remark}
\newtheorem{remark}{Remark}

\title
[
	Well-posedness for a Whitham--Boussinesq system
]
{
	Well-posedness for a Whitham--Boussinesq system
	with surface tension
}

\author{ Evgueni Dinvay }

\email{ Evgueni.Dinvay@uib.no }
 
\address
{
	Department of Mathematics
	\\
	University of Bergen
	\\
	PO Box 7803
	\\
	5020 Bergen
	\\
	Norway
}

\subjclass[2010]{35Q53, 35Q55, 35A01} 

\begin{document}

\begin{abstract} 
We regard the Cauchy problem for
a particular Whitham--Boussinesq system
modelling surface waves
of an inviscid incompressible fluid layer.
The system can be seen as a weak nonlocal
dispersive perturbation of the shallow water system.
The proof of well-posedness relies
on energy estimates.
However, due to the symmetry lack of the nonlinear part,
in order to close the a priori estimates
one has to modify the traditional energy norm in use.
Hamiltonian conservation provides with global
well-posedness at least for small initial data
in the one dimensional settings.
\end{abstract}

\maketitle

\section{Introduction}
\setcounter{equation}{0}

Consideration is given to the following one-dimensional
Whitham-type system
\begin{equation}
\label{capillarity_sys}
\left\{
\begin{aligned}
	\partial_t \eta &=
	- \partial_x v - i \tanh D (\eta v)
	, \\
	\partial_t v &=
	- i \tanh D (1 + \varkappa D^2) \eta - i \tanh D v^2 / 2
	,
\end{aligned}
\right.
\end{equation} 
where $D=-i\partial_x$
and $\tanh D$ are Fourier multiplier operators in
the space of tempered distributions $\mathcal S'(\mathbb R)$.
The positive parameter $\varkappa$ stands for the surface tension here.
The space variable is $x \in \mathbb R$ and the time variable
is $t \in \mathbb R$.
The unknowns $\eta$, $v$ are real valued functions of these
variables.
We pick the initial values $\eta(0)$, $v(0)$
corresponding to the time moment $t = 0$
in Sobolev spaces as follows
\begin{equation}
\label{capillarity_data}
	\eta(0) = \eta_0 \in H^{s+1/2}(\R)
	, \qquad
	v(0) = v_0 \in H^s(\R)
	,
\end{equation} 
where $s \geqslant 1/2$.
System \eqref{capillarity_sys}
has the Hamiltonian structure
\[
	\partial_t (\eta, v)^T = \mathcal J \nabla \mathcal H(\eta, v)
\]
with the skew-adjoint matrix
\[
	\mathcal J
	=
	\begin{pmatrix}
		0 & - i \tanh D
		\\
		- i \tanh D & 0
	\end{pmatrix}
\]
and the energy functional
\begin{equation}
\label{Hamiltonian}
	\mathcal H(\eta, v)  = \frac 12 \int
	\left(	
		\eta^2 + \varkappa (\partial_x \eta)^2
		+ v \frac{D}{\tanh D} v
		+ \eta v^2
	\right)
	dx
\end{equation}
well defined on $H^1 \times H^{1/2}$.
The latter conserves on solutions together with momentum
$\mathcal I(\eta, v)$ that has the same view
as in the pure gravity case
\begin{equation*}
	\mathcal I(\eta, v)  = \int
	\eta \frac{D}{\tanh D} v dx
	.
\end{equation*}

In case of the trivial surface tension $\varkappa = 0$,
System \eqref{capillarity_sys} was proposed
in \cite{Dinvay_Dutykh_Kalisch}
as an approximate model for the study of water waves
to provide a two-directional alternative to
the well-known Whitham equation \cite{Whitham}.
The latter was proved to be consistent
with the KdV equation \cite{Klein_Linares_Pilod_Saut}
in the long wave regime \cite{Lannes}.
We also refer to \cite{Duchene_Israwi_Talhouk} for another
version of the fully-dispersive Boussinesq type.
Importance of such models is supported
by experiments \cite{Carter}.
The unknown $\eta$ denotes the deflection
of the free surface from its equilibrium position,
corresponding to the vertical level $z = 0$.
The bottom is assumed to be flat and located
at the level $z = -1$.
The variable $v$ is associated with the free surface velocity
as explained in \cite{Dinvay_Dutykh_Kalisch}.

The initial value problem
for Model \eqref{capillarity_sys}
was studied in \cite{Dinvay, Dinvay_Tesfahun}
in the case of vanishing surface tension $\varkappa = 0$.
In the same framework existence of solitary waves
was proved in \cite{Dinvay_Nilsson}.
A natural extension of the existing results is
to consider the case of non-trivial
capillarity $\varkappa > 0$.
Note that the term $1 + \varkappa D^2$ could be applied
to $-v_x$ in the first equation instead,
as it is done in \cite{Kalisch_Pilod},
for example,
to regularise the system regarded
in \cite{Pei_Wang}.
However, the case regarded here is physically more relevant
\cite{Dinvay_Moldabayev}.
Indeed, repeating the Hamiltonian perturbation analysis
from \cite{Dinvay_Dutykh_Kalisch} to the full Hamiltonian
with the surface tension, that can be found in
\cite{Dinvay_Moldabayev},
one naturally arrives to
\eqref{capillarity_sys}, \eqref{Hamiltonian}.
It turns out that surface tension changes the nature of the equations.
Indeed, the multiplication operator $(\eta, v) \mapsto v \eta$
is not bounded in the natural Sobolev-based energy space arising
from the linear equations, that is
\(
	H^{s + 1/2}(\mathbb{R}) \times H^s(\mathbb{R})
	.
\)
And moreover,  the dispersive properties of
the corresponding linear semigroup are insufficient to counterbalance
the loss of 1/2 derivatives.
As a result the proof of well-posedness demands a technique
different from the one used in \cite{Dinvay_Tesfahun}.

As to additional initial conditions, apart from
inclusions given in \eqref{capillarity_data}, one
has to impose a restriction essentially similar
to the one used in \cite{Dinvay_Tesfahun},
namely, smallness of the 
$H^1 \times H^{1/2}$-norm of $(\eta_0, v_0)$.
This is important for the global-in-time existence.
The meaning of this condition is that the
total energy
$\mathcal H(\eta_0, v_0)$ should be positive and not too big.
We point out that this condition cannot be significantly
weakened even for the proof of the local result,
which is also different from the non-capillarity situation.
More precisely, for the local regular
($s$ is large enough in \eqref{capillarity_data})
well-posedness result
it is enough to assume non-cavitation instead.

\begin{definition}
\label{noncavitation_definition}
	Let $d = 1,2$.
	We say that elevation
	\(
		\eta \in C \left( [0, T];
		L^{\infty}\left( \mathbb R^d \right) \right)
	\)
	satisfies the non-cavitation condition
	if there exist $h, H > 0$ such that
	$H \geqslant \eta \geqslant h - 1$ on $\mathbb R^d \times [0, T]$.
	Analogously,
	one defines non-cavitation at a particular time moment.
\end{definition}

The non-cavitation condition
is a physical condition meaning that
the elevation of the wave should not touch the bottom of the fluid
for System \eqref{capillarity_sys} to be a relevant model.
For convenience we have also included boundedness from above
in this definition.
We exploit the definition for providing with more general
local existence formulation at high regularity level.
However, in the low regularity case this condition
cannot be controlled without imposing a stronger assumption,
as we shall see below.
We turn now to the formulation of the main results.

\begin{theorem}
\label{capillarity_theorem_local}
	Let $s > 3/2$.
	Suppose that the initial data \eqref{capillarity_data}
	satisfies the noncavitation condition.
	Then there exist $T > 0$ and a unique solution
	\[
		( \eta, v ) \in C \left( [0, T];
		H^{s + 1/2}(\mathbb{R}) \times H^s(\mathbb{R}) \right)
		\cap C^1 \left( (0, T);
		H^{s - 1}(\mathbb{R}) \times H^{s - 3/2}(\mathbb{R}) \right)
	\]
	of System \eqref{capillarity_sys}
	with the initial data $( \eta_0, v_0 )$.
	The time moment $T$ depends on $s$, $\varkappa$ and
	the norm
	\(
		\norm{ \eta_0, v_0 } _{ H^{s + 1/2} \times H^s }
		.
	\)
	With respect to the capillarity and the initial data norm,
	the time of existence $T$ is a non-increasing function.
	Moreover, the solution depends continuously on
	the initial data with respect to
	\(
		C \left( H^{s + 1/2} \times H^s \right)
	\)-norm.
\end{theorem}

It is worth to emphasize here that
the time of existence does not shrink as the surface
tension parameter goes to zero.
Making a bit stronger assumption on the initial data \eqref{capillarity_data},
one obtains a stronger result.

\begin{theorem}
\label{capillarity_theorem}
	Let $s > 1/2$.
	For any $\eta_0 \in H^{s + 1/2}(\mathbb{R})$
	and
	$v_0 \in H^s(\mathbb{R})$
	having sufficiently small
	$H^1 \times H^{1/2}$-norm
	there exists a unique global solution
	\[
		( \eta, v ) \in C \left( [0, \infty);
		H^{s + 1/2}(\mathbb{R}) \times H^s(\mathbb{R}) \right)
		\cap C^1 \left( (0, \infty);
		H^{s - 1}(\mathbb{R}) \times H^{s - 3/2}(\mathbb{R}) \right)
	\]
	of System \eqref{capillarity_sys}
	with the initial data $( \eta_0, v_0 )$.
	Moreover, the solution depends continuously on
	the initial data with respect to
	\(
		C \left( H^{s + 1/2} \times H^s \right)
	\)-norm
	on any finite time interval $[0, T]$.
\end{theorem}

As we shall see below, the smallness of $H^1 \times H^{1/2}$-norm
plays an essential role in proving the following two statements.
The Cauchy problem \eqref{capillarity_sys}, \eqref{capillarity_data}
is locally well-posed for $1/2 < s \leqslant 3/2$.
The solution can be extended to the global one for any $s > 1/2$.
Whereas for the local result in the case $s > 3/2$,
it is enough to impose a weaker assumption, namely,
the noncavitation of $\eta_0$.

The proof is essentially based on the energy method,
that is natural to apply to quasilinear equations.
The scaling
\(
	H^{s + 1/2}(\mathbb{R}) \times H^s(\mathbb{R})
\)
is needed to rule out the linear terms,
after the differentiation the corresponding energy norm.
The main difficulty lies in the lack of symmetry
of the nonlinearity.
Indeed, a direct time differentiation
of the norm
\(
	\norm{\eta, v} _{ H^{s + 1/2} \times H^s }
\)
leads to the term
\(
	\int \left( J^{s - 1/2} \partial_x \eta \right)
	\eta J^{s + 1/2} v
	,
\)
where $J^{\sigma}$ stands for the Bessel potential
of order $-\sigma$
(see the proof of Lemma \ref{energy_lemma} below).
Note that this term cannot be handled
by integration by parts or commutator estimates, and so cannot be
estimated via the energy norm.
To overcome this difficulty we modify the energy norm
adding the cubic term
\(
	\int \eta \left( J^{s - 1/2 } v \right) ^2
	.
\)
The linear contribution of the derivative of this term
will cancel out the mentioned inconvenient term.
Meanwhile, the contribution coming from the nonlinear terms
can easily be controlled.
As we point out below a hint on the choice of the modifier
comes from Hamiltonian \eqref{Hamiltonian}.
Note that after adding the cubic term the energy loses
coercivity, and so one has to impose an additional condition.
Either the noncavitation for big $s$ or the smallness for small $s$
of the initial data, both propagating through the flow of
System \eqref{capillarity_sys}, is enough to ensure
that the modified energy is coercive.

%
Additionally, consideration is also given to
a system posed on ${\R^{2+1}}$ of
the form
\begin{equation}
\label{wt2d}
\left\{
\begin{aligned}
  	\partial_t \eta + \nabla \cdot \mathbf v
  	& =
	- K^2 \nabla \cdot (\eta  \mathbf v)
	,
  	\\
  	\partial_t  \mathbf v + K^2 \nabla (1 + \varkappa |D|^2) \eta
 	&=
 	- K^2 \nabla \left( | \mathbf v|^2/2 \right)
 	,
\end{aligned}
\right.
\end{equation}
that is a direct two dimensional extension of
Model \eqref{capillarity_sys}.
Here $ \mathbf v= (v_1, v_2)\in \R^2$
is a curl free vector field, that is
$\nabla \times  \mathbf v   = 0$, and 
$$
	K = \sqrt{ \tanh|D| / |D|}
$$
with $D=-i \nabla$.
So the corresponding symbol
$
	K(\xi)
	=
	\sqrt{
		\tanh (|\xi|) / |\xi|
	}
	.
$
We complement \eqref{wt2d} with the initial data 
\begin{equation}
\label{data2d}
\eta(0)= \eta_0 \in H^{s+1/2} \left( \R^2 \right)
, \qquad
\mathbf v(0)= \mathbf v_0
\in H^s \left( \R^2 \right) \times H^s \left( \R^2 \right)
.
\end{equation} 
As above the variables $\eta$ and $\mathbf v$
stand for the surface elevation
and the surface fluid velocity, respectively.
The system enjoys the Hamiltonian structure
\[
	\partial_t (\eta,  \mathbf  v)^T
	=
	\mathcal J \nabla \mathcal H(\eta ,  \mathbf v)
\]
with the skew-adjoint matrix
\[
	\mathcal J
	=
	\begin{pmatrix}
		0 & - K^2 \partial_{x_1} & - K^2 \partial_{x_2}
		\\
		- K^2 \partial_{x_1} & 0 & 0
		\\
		- K^2 \partial_{x_2} & 0 & 0
	\end{pmatrix}
	,
\]
which in particular,
guarantees conservation of the energy functional
\begin{equation}
\label{Hamiltonian2}
	\mathcal H(\eta, \mathbf v)  = \frac 12 \int
	\left(	
		\eta^2 + \varkappa |\nabla \eta|^2
		+ \left| K^{-1} \mathbf v \right|^2
		+ \eta |\mathbf v |^2
	\right)
	dx
	.
\end{equation}
The noncavitation definition in the two dimensional problem
has exactly the same view
as in Definition \ref{noncavitation_definition} with the real line
$\mathbb R$ substituted by the plane $\mathbb R^2$.

\begin{theorem}
\label{capillarity_theorem2d}
	Let $s > 1$.
	Suppose that the initial data \eqref{data2d}
	has curl free velocity $\nabla \times  \mathbf v_0 = 0$
	and
	either has small enough
	$H^1 \times H^{1/2} \times H^{1/2}$-norm
	if $s \leqslant 2$ or
	satisfies the noncavitation condition
	if $s > 2$.
	Then there exist $T > 0$
	and a unique solution
	\[
		( \eta, \mathbf v ) \in C \left( [0, T];
		H^{s + 1/2} \left( \mathbb R^2 \right) \times
		\left( H^s \left( \mathbb R^2 \right) \right)^2 \right)
		\cap C^1 \left( (0, T);
		H^{s - 1} \left( \mathbb R^2 \right) \times
		\left( H^{s - 3/2} \left( \mathbb R^2 \right) \right)^2 \right)
	\]
	of System \eqref{wt2d}
	associated with this initial data.
	The time of existence $T$ is a non-increasing
	function of the surface tension $\varkappa$ and
	the initial data norm
	\(
		\norm{ \eta_0, \mathbf v_0 } _{ H^{s + 1/2} \times H^s \times H^s }
		.
	\)
	Moreover, the solution depends continuously on
	the initial data with respect to
	\(
		C \left( H^{s + 1/2} \times H^s \times H^s \right)
	\)-norm.
\end{theorem}

Note that the theorem has
the local character, in the opposite of the one dimensional case.
%

\begin{remark}
	The same results hold in the periodic case as well.
	The proof is similar up to some small changes
	in the commutator estimates
	\cite{Kenig_Pilod}.
\end{remark}

In the next section some important inequalities are recalled.
In Section \ref{Modified_energy} we introduce
the modified energy and obtain the corresponding energy estimate
for System \eqref{capillarity_sys}.
In Section \ref{Uniqueness_type_estimate}
we obtain the energy estimate for the difference
of two solutions of System \eqref{capillarity_sys}.
Note that Sections \ref{Modified_energy},
\ref{Uniqueness_type_estimate}
provide with the motivation for studying
the parabolic regularisation later in
Section \ref{Parabolic_regularisation}, 
where the corresponding
energy estimate is deduced for the regularised system.
In Section \ref{A_priori_estimate} a priori estimates are obtained.
Finally, in Section \ref{Proof_of_Theorem} we comment on the last
steps in the proof of Theorem \ref{capillarity_theorem},
omitting only the thorough discussion of
the initial data regularisation.
%
In Section \ref{Two_dimensional_problem} we discuss some peculiarities
of the two dimensional problem.
%
In the last section we study System \eqref{capillarity_sys}
with $\varkappa \ll 1$.

\section{Preliminary estimates}
\setcounter{equation}{0}

We start this section by recalling
all the necessary standard notations.
For any positive numbers $a$ and $b$ we write
$a \lesssim b$ if there exists a constant $C$ independent
of $a, b$ such that $a \leqslant Cb$.
The Fourier transform is defined by the formula
\[
	\widehat{f}(\xi) = \mathcal F(f)(\xi) =
	\int	 f(x) e^{-i\xi x} dx
\]
on Schwartz functions.
By the Fourier multiplier operator $\varphi(D)$
with symbol $\varphi$ we mean
the line
\(
	\mathcal F \left( \varphi(D) f \right)
	=
	\varphi(\xi) \widehat{f}(\xi)
	.
\)
In particular, $D = -i\partial_x$ is
the Fourier multiplier associated with the symbol
$\varphi(\xi) = \xi$.
For any $\alpha \in \mathbb R$
the Riesz potential of order $-\alpha$ is
the Fourier operator $|D|^{\alpha}$
and
the Bessel potential of order $-\alpha$ is
the Fourier operator $J^{\alpha} = \langle D \rangle ^{\alpha}$,
where we exploit the notation
$\langle \xi \rangle = \sqrt{1 + \xi^2}$.
The $L^2$-based Sobolev space $H^{\alpha} (\mathbb R)$ is defined
by the norm
\(
	\norm{f} _{H^{\alpha}}
	=
	\norm{J^{\alpha}f} _{L^2}
	,
\)
whereas
the homogeneous Sobolev space $\dot H^{\alpha} (\mathbb R)$ is defined
by
\(
	\norm{f} _{\dot H^{\alpha}}
	=
	\norm{|D|^{\alpha}f} _{L^2}
	.
\)
We also exploit the notation
\(
	H^{\infty} (\mathbb R)
	=
	\cap _{\alpha \in \mathbb R} H^{\alpha} (\mathbb R)
	.
\)

Introduce the operator
\begin{equation}
\label{K_kappa}
	K _{ \varkappa }
	= \sqrt{( 1 + \varkappa |D|^2 ) \frac{\tanh |D|}{|D|} }
	,
\end{equation}
where $\varkappa$ is the surface tension.
Note that $\varkappa > 0$ is a fixed constant.
We implement the notation $K = K_0 = \sqrt{\tanh D / D}$
used in \cite{Dinvay_Tesfahun}.
Its inverse
$K^{-1}$ and $K _{ \varkappa }$ both have the domain
$H^{1/2} (\mathbb R)$
and are equivalent to the Bessel potential $J^{1/2}$.
Below we will need to compare $J$, $|D|$ and $K^{-2}$
and so we prove the following simple estimates.

\begin{lemma}
\label{Bessel_Riesz_comparison_lemma}
	For any $f \in L^2(\mathbb R)$ it holds that
	\[
		\left \lVert
			\left( J - K^{-2} \right) Df
		\right \rVert _{L^2}
		\leqslant
		\norm{ (J - |D|) Df } _{L^2}
		\leqslant
		\frac 12 \norm{ f } _{L^2}
		.
	\]
\end{lemma}

\begin{proof}
By the Plancherel identity it is enough to check
the following inequalities
\[
	0 \leqslant
	\langle \xi \rangle - \frac{\xi}{\tanh \xi}
	\leqslant
	\langle \xi \rangle - |\xi|
	\leqslant
	\frac 1{2|\xi|}
	,
\]
where the middle one is trivial.
The rightmost inequality follows from
\[
	\langle \xi \rangle - |\xi|
	=
	\frac{1}{\langle \xi \rangle + |\xi|}
	\leqslant
	\frac 1{2|\xi|}
	.
\]
The leftmost one follows from
the $\tanh$-definition via exponents and the obvious
\[
	e^{2\xi} + e^{-2\xi} \geqslant
	2 + 4\xi^2
	.
\]
\end{proof}

Throughout the text we make an extensive use of the following
bilinear estimates.
Firstly, we state the Kato-Ponce commutator estimate \cite{Kato_Ponce}.

\begin{lemma}[Kato-Ponce commutator estimate]
\label{Kato_Ponce_lemma}
Let $s \geqslant 1$, $p,\ p_2, \ p_3 \in (1,\infty)$
and $p_1, \ p_4 \in (1,\infty]$ be such that 
$\frac1p=\frac1{p_1}+\frac1{p_2}=\frac1{p_3}+\frac1{p_4}$.
Then 
\begin{equation}
\label{Kato_Ponce}
\|[J^s,f]g\|_{L^p} \lesssim \|\partial_xf\|_{L^{p_1}}\|J^{s-1}g\|_{L^{p_2}}+\|J^sf\|_{L^{p_3}}\|g\|_{L^{p_4}}
\end{equation}
for any $f, \, g $ defined on $\mathbb R$.
\end{lemma}

By the commutator $[A, B]$ between operators
$A$ and $B$ we mean the operator
\(
	[A, B] f = ABf - BAf
	.
\)
Secondly, we state the fractional Leibniz rule
proved in the appendix of \cite{Kenig_Ponce_Vega_Communications}.

\begin{lemma}
\label{LeibnizRule_lemma} 
Let $\sigma =\sigma_1+\sigma_2 \in (0,1)$
with $\sigma_i \in (0,\sigma)$
and $p, \  p_1, \ p_2 \in (1,\infty)$ satisfy
$\frac1p=\frac1{p_1}+\frac1{p_2}$.
Then
\begin{equation}
\label{LeibnizRule}
\left \lVert
	|D|^{\sigma}(fg)-f|D|^{\sigma}g-g|D|^{\sigma}f
\right \rVert_{L^p}
\lesssim
\||D|^{\sigma_1}f\|_{L^{p_1}}\||D|^{\sigma_2}g\|_{L^{p_2}}
\end{equation}
for any $f, \, g $ defined on $\mathbb R$.
Moreover, the case $\sigma_2=0$, $p_2=\infty$ is also allowed.
\end{lemma}

We also state an estimate, firstly appeared
in \cite{Klainerman_Selberg} in a weaker form,
and later sharpened in \cite{Selberg_Tesfahun}.

\begin{lemma}
\label{Selberg_Tesfahun_lemma}
	Suppose $a, b, c \in \mathbb R$.
	Then for any $f \in H^a(\mathbb R)$,
	$g \in H^b(\mathbb R)$ and $h \in H^c(\mathbb R)$
	the following inequality holds
	\begin{equation}
	\label{Selberg_Tesfahun}
		\lVert fgh \rVert _{L^1}
		\lesssim
		\lVert f \rVert _{H^a}
		\lVert g \rVert _{H^b}
		\lVert h \rVert _{H^c}
	\end{equation}
	provided that
	\[
		a + b + c > \frac 12
		,
	\]
	\[
		a + b \geqslant 0
		,
		\quad
		a + c \geqslant 0
		,
		\quad
		b + c \geqslant 0
		.
	\]
\end{lemma}

Proving a global-in-time a priori estimate we will use
the following limiting case of the Sobolev embedding theorem,
that in the one dimensional case $d = 1$
reads as follows.

\begin{lemma}
[Brezis-Gallouet inequality]
\label{Brezis_lemma}
	Suppose $f \in H^s \left( \mathbb R \right)$ with $s > 1 / 2$.
	Then
	\begin{equation}
	\label{Brezis_inequality}
		\lVert f \rVert_{L^{\infty}}
		\leqslant
		C_{s}
		\left(
			1 + \lVert f \rVert_{H^{1/2}}
			\sqrt{ \log( 1 + \lVert f \rVert_{H^s} ) }
		\right)
		.
	\end{equation}
\end{lemma}


Inequality \eqref{Brezis_inequality} was firstly put forward
and proved in $H^2(\mathbb R^2)$
in the work by Brezis, Gallouet
\cite{Brezis_Gallouet}.
It was extended to more general Sobolev spaces
and any dimension in
\cite{Brezis_Wainger}, but in a slightly different form.
For the sake of completeness, we provide here with
the proof based on the idea introduced in \cite{Brezis_Gallouet}.

\begin{proof}

Let $f \in H^s(\mathbb R)$ with $s > 1 / 2$.
Then
\[
	\norm{f}_{L^{\infty}}
	\leqslant
	\frac 1{2\pi} \norm{ \widehat{f} }_{L^1}
	=
	\frac 1{2\pi} \int _{|\xi| \leqslant R}
	\left| \widehat{f} (\xi) \right| d\xi
	+
	\frac 1{2\pi} \int _{|\xi| > R}
	\left| \widehat{f} (\xi) \right| d\xi
	= I_1(R) + I_2(R)
	,
\]
where $R > 0$ is an arbitrary positive number.
In the first integral $I_1$ we multiply and divide
$\widehat{f}$ by $\left (1 + \xi^2 \right) ^{1/4}$.
Afterwards, we apply the H{\"o}lder inequality
to get
\[
	I_1(R)
	\leqslant
	\frac 1{2\pi}
	\left(
		\int_{-R}^R 
		\left| \widehat{f} (\xi) \right|^2
		\left( 1 + \xi^2 \right)^{1/2} d\xi
	\right) ^{1/2}
	\left(
		\int_{-R}^R \frac{d\xi}{\sqrt{1 + \xi^2}}
	\right) ^{1/2}
	\leqslant
	\frac 1{\pi} \norm{ f }_{H^{1/2}}
	\sqrt{ \log (1 + R) }
	,
\]
and similarly,
\[
	I_2(R)
	\leqslant
	\frac 1{2\pi} \norm{ f }_{H^s}
	\left(
		\int_{|\xi| > R} \frac{d\xi}{ \left( 1 + \xi^2 \right)^s}
	\right) ^{1/2}
	.
\]
Now it is left to choose $R$ depending on $f, s$.
If
\(
	\norm{ f }_{H^s} \leqslant 1
\)
then taking
\(
	R = \norm{ f }_{H^s}
\)
we immediately obtain the desired inequality \eqref{Brezis_inequality}.
In the case
\(
	\norm{ f }_{H^s} > 1
	,
\)
we estimate the second integral as follows
\[
	I_2(R)
	\leqslant
	\frac { \norm{ f }_{H^s} }
	{ 2\pi \sqrt{s - 1/2} R^{s - 1/2} }
	,
\]
and so if
\(
	s \geqslant 3/2
\)
one takes 
\(
	R = \norm{ f }_{H^s}
\)
again to bound 
\(
	I_2(R)
	\leqslant
	\left( 2\pi \sqrt{s - 1/2} \right)^{-1}
\)
and to come to \eqref{Brezis_inequality}.
In the last case
\(
	\norm{ f }_{H^s} > 1
\)
and
\(
	\alpha_s = 1 / (s - 1/2) > 1
	,
\)
we can take
\(
	R = \norm{ f }_{H^s}^{\alpha_s}
\)
to bound $I_2(R)$ by the same constant.
Note that $I_1(R)$ in this case is bounded as
\[
	I_1(R)
	\leqslant
	\frac 1{\pi} \norm{ f }_{H^{1/2}}
	\sqrt{ \log \left( 1 + \norm{ f }_{H^s}^{\alpha_s} \right) }
	\leqslant
	\frac 1{\pi} \norm{ f }_{H^{1/2}}
	\sqrt{ \alpha_s \log \left( 1 + \norm{ f }_{H^s}\right) }
	,
\]
and so we again obtain Inequality \eqref{Brezis_inequality}.
 
\end{proof}

\section{Modified energy}
\label{Modified_energy}
\setcounter{equation}{0}

As we shall see in the proof of the next lemma,
a direct use of 
\(
	H^{s + 1/2} \times H^s
\)-norm as the energy
does not allow us to close the estimates,
and so we modify it as follows.
Firstly, for each $\varkappa > 0$ and $s \geqslant 1/2$
we introduce the norm
\begin{equation}
\label{capillarity_norm}
	\norm{ \eta, v } _{ H_{\varkappa}^{s + 1/2} \times H^s }^2
	= \varkappa \lVert \partial_x \eta \rVert _{H^{s - 1/2}}^2
	+ \lVert \eta \rVert _{H^{s - 1/2}}^2
	+ \lVert K^{-1} v \rVert _{H^{s - 1/2}}^2
	,
\end{equation}
which is obviously equivalent to the standard norm in
\(
	H^{s +1/2}(\mathbb{R}) \times H^s(\mathbb{R})
	.
\)
Such choice will be convenient later for
analysis of dependence of solution on the
capillarity $\varkappa$.
Secondly, we define the modified energy
\begin{equation}
\label{energy}
	E^s(\eta, v)
	= \frac 12
	\norm{ \eta, v } _{ H_{\varkappa}^{s + 1/2} \times H^s }^2
	+ \frac 12 \int \eta \left( J^{s - 1/2 } v \right) ^2
	,
\end{equation}
where the pair $(\eta, v)$ represents a possible solution of
System \eqref{capillarity_sys}.
Note that in the limit case $s = 1/2$ this quantity
coincides with the Hamiltonian given in \eqref{Hamiltonian},
\(
	E^{1/2}(\eta, v)
	= \mathcal H(\eta, v)
	.
\)
This gives us a small hint for the choice of the right cubic modifier
that is basically a guess.

\begin{lemma}
\label{energy_lemma}
	Suppose $s \geqslant 1/2$.
	Then there exists $C_s > 0$ such that for any $\varkappa > 0$ and
	any functions
	\(
		\eta, v \in C^1 \left( (0, T);
		H^{\infty}(\mathbb{R}) \right)
	\)
	solving System \eqref{capillarity_sys}
	it holds that
	\[
		\frac{d}{dt}E^s (\eta, v)
		\leqslant
		C_s (1 + \varkappa)
		\left(
			\norm{ \eta, v } _{ H_{\varkappa}^{s + 1/2} \times H^s }^2
			+
			\norm{ \eta, v } _{ H_{\varkappa}^{s + 1/2} \times H^s }^4
		\right).
	\]
\end{lemma}
\begin{proof}
We have already noticed that $E^{1/2}(\eta, v)$
is a conserved quantity, which proves the
statement for the limit case $s = 1/2$.

Assuming $s > 1/2$ we
calculate the derivatives
%
\[
	\frac {\varkappa}2 \frac{d}{dt}
	\lVert \partial_x \eta \rVert _{H^{s-1/2}}^2
	=
	-\varkappa \int \left( J^{s - 1/2} \partial_x \eta \right)
	J^{s - 1/2} \partial_x^2 v
	- i\varkappa \int \left( J^{s - 1/2} \partial_x \eta \right)
	J^{s - 1/2} \partial_x \tanh D(\eta v)
	,
\]
%
\[
	\frac 12 \frac{d}{dt}
	\lVert \eta \rVert _{H^{s-1/2}}^2
	=
	- \int \left( J^{s - 1/2} \eta \right)
	J^{s - 1/2} \partial_x v
	- i \int \left( J^{s - 1/2} \eta \right)
	J^{s - 1/2} \tanh D(\eta v)
\]
and the derivative of velocity norm
\begin{multline*}
	\frac 12 \frac{d}{dt} \norm{ K^{-1} v } _{H^{s - 1/2}}^2
	=
	- i \int ( J^{s - 1/2} K^{-1} v ) J^{s - 1/2} K^{-1}
	\tanh D \left( 1 + \varkappa D^2 \right) \eta
	\\
	- \frac i2 \int ( J^{s - 1/2} K^{-1} v )
	J^{s - 1/2} K^{-1} \tanh D v^2
	.
\end{multline*}
Summing up these derivatives and
simplifying the corresponding expression via
integration by parts,
we obtain
\[
	\frac 12 \frac{d}{dt}
	\norm{ \eta, v } _{ H_{\varkappa}^{s + 1/2} \times H^s }^2
	=
	I_1 + I_2 + I_3
	,
\]
where
\[
	I_1 =
	i\varkappa \int \left( J^{s - 1/2} D^2 \tanh D \eta \right)
	J^{s - 1/2} (\eta v)
	,
\]
\[
	I_2 =
	i \int \left( J^{s - 1/2} \tanh D \eta \right)
	J^{s - 1/2} (\eta v)
	,
\]
\[
	I_3 =
	\frac i2 \int \left( J^{s - 1/2} |D|^{1/2} \sgn D v \right)
	J^{s - 1/2} |D|^{1/2} (v^2)
	.
\]
The second integral $I_2$ can be estimated with the help of
Lemma \ref{Selberg_Tesfahun_lemma}, by
setting
\(
	f = J^{2s - 1} \tanh D \eta
	,
\)
$g = \eta$, $h = v$ and $a = 1/2 - s$, $b = s - 1/2$,
$c = s$.
Thus one obtains
\[
	I_2 \lesssim
	\lVert \eta \rVert _{H^{s - 1/2}}^2
	\lVert v \rVert _{H^s}
	.
\]
Applying H\"older's inequality to the third integral $I_3$,
we get
\[
	I_3
	\lesssim
	\norm { v } _{H^s}
	\norm { v^2 } _{H^s}
	\lesssim
	\norm { v } _{H^s}^2
	\norm { v } _{L^{\infty}}
	.
\]
We would like to point out here that
the first integral $I_1$ cannot be
estimated via the energy norm \eqref{capillarity_norm},
using only integration by parts or commutator estimates.
Turning our attention to the modifier of energy $E^s$,
we calculate its time derivative as follows
\begin{multline}
\label{dE3}
	\frac 12 \frac {d}{dt}
	\int \eta \left( J^{s - \frac 12 } v \right) ^2
	=
	- i \int \eta \left( J^{s - \frac 12 } v \right)
	J^{s - \frac 12 } \tanh D \eta
	- i \varkappa \int \eta \left( J^{s - \frac 12 } v \right)
	J^{s - \frac 12 } D^2 \tanh D \eta
	\\
	- \frac i2 \int \eta \left( J^{s - \frac 12 } v \right)
	J^{s - \frac 12 } \tanh D v^2
	- \frac 12
	\int \partial_x v \left( J^{s - \frac 12 } v \right) ^2
	- \frac i2
	\int \tanh D (\eta v) \left( J^{s - \frac 12 } v \right) ^2
	.
\end{multline}
Let $I_4, \ldots, I_8$ represent these integrals, respectively.
The first summand, that we notate by $I_4$,
is estimated easily as
\[
	I_4 =
	- i \int \eta \left( J^{s - \frac 12 } v \right)
	J^{s - \frac 12 } \tanh D \eta
	\lesssim
	\lVert \eta \rVert _{H^{s - 1/2}}^2
	\lVert v \rVert _{H^s}
\]
by applying Inequality \eqref{Selberg_Tesfahun}.
The third integral in \eqref{dE3},
notated by $I_6$,
is estimated in a similar way
\[
	I_6
	\lesssim
	\lVert \eta \rVert _{L^2}
	\norm{ v } _{H^s}
	\norm{ v^2 } _{H^s}
	\lesssim
	\lVert \eta \rVert _{L^2}
	\norm{ v } _{L^{\infty}}
	\norm{ v } _{H^s}^2
	\lesssim
	\lVert \eta \rVert _{H^{s - 1/2}}
	\lVert K^{-1} v \rVert _{H^{s - 1/2}}^3
	.
\]
The fourth integral in \eqref{dE3}
equals
\begin{multline*}
	I_7
	=
	- \frac i2
	\int \left( \sgn D |D|^{\frac 12} v \right)
	|D|^{\frac 12} \left( J^{s - \frac 12 } v \right) ^2
	=
	- i \int \left( \sgn D |D|^{\frac 12} v \right)
	\left( J^{s - \frac 12 } v \right)
	J^{s - \frac 12 } |D|^{\frac 12} v
	\\
	- \frac i2
	\int \left( \sgn D |D|^{\frac 12} v \right)
	\left[
		|D|^{\frac 12} \left( J^{s - \frac 12 } v \right) ^2
		- 2 \left( J^{s - \frac 12 } v \right)
		J^{s - \frac 12 } |D|^{\frac 12} v
	\right]
\end{multline*}
where the first integral can be treated
with interpolation in Sobolev spaces and
the second integral by the fractional Leibniz rule
as follows
\[
	I_7 \lesssim
	\lVert \sgn D |D|^{\frac 12} v \rVert _{H^{s - 1/2}}
	\lVert J^{s - \frac 12 } v \rVert _{H^{1/2}}
	\lVert J^{s - \frac 12 } |D|^{\frac 12} v \rVert _{L^2}
	+
	\lVert \sgn D |D|^{\frac 12} v \rVert _{L^2}
	\lVert J^{s - \frac 12 } |D|^{\frac 12} v \rVert _{L^2}^2
	\lesssim
	\lVert v \rVert _{H^s}^3.
\]
The last integral in \eqref{dE3},
that we notate by $I_8$, is bounded by
\[
	I_8 \leqslant
	\frac 12
	\lVert \eta \rVert _{L^2}
	\lVert v \rVert _{L^{\infty}}
	\lVert J^{s - 1/2 } v \rVert _{L^4}^2
	\lesssim
	\lVert \eta \rVert _{H^{s - 1/2 }}
	\lVert K^{-1} v \rVert _{H^{s - 1/2}}^3.
\]
It is left to regard the second integral in \eqref{dE3},
denoted by $I_5$,
and the integral $I_1$ appeared after the differentiation
of the energy norm \eqref{capillarity_norm}.
Firstly, let us note that
\[
	J^{s + 1/2} (\eta v)
	=
	\left[ J^{s + 1/2}, \eta \right]v
	+ \eta J^{s + 1/2} v
	,
\]
\[
	J \left( \eta J^{s - 1/2} v \right)
	=
	\left[ J, \eta \right] J^{s - 1/2} v
	+ \eta J^{s + 1/2} v
	,
\]
and so summing $I_1$, $I_5$ together one can easily obtain
\[
	I_1 + I_5 =
	i\varkappa \int \left( J^{s - 3/2} D^2 \tanh D \eta \right)
	\left(
		\left[ J^{s + 1/2}, \eta \right]v
		-
		\left[ J, \eta \right] J^{s - 1/2} v
	\right)
	.
\]
Applying the Kato-Ponce estimate
to the first commutator one obtains
\[
	\norm{ \left[ J^{s + 1/2}, \eta \right]v }_{L^2}
	\lesssim
	\lVert \partial_x \eta \rVert _{L^{p_1}}
	\lVert J^{s - 1/2} v \rVert _{L^{p_2}}
	+
	\lVert J^{s + 1/2} \eta \rVert _{L^2}
	\lVert v \rVert _{L^{\infty}}
	.
\]
Taking $p_1(s) = \frac{1}{1-s}$, $p_2(s) = \frac{2}{2s-1}$
for $s \in (\frac 12, 1)$
and $p_1 = p_2 = 4$ in case $s \geqslant 1$
one deduces
\[
	\norm{ \left[ J^{s + 1/2}, \eta \right]v }_{L^2}
	\lesssim
	\norm{ \eta } _{H^{s + 1/2}}
	\left\{
	\begin{aligned}
			\norm{ v } _{H^{1/2}}
			+ \norm{ v } _{L^{\infty}}
		&
		\ \mbox{ for }
		s \in (1/2, 1)
		\\
		\norm{ v } _{H^{s - 1/4}}
		&
		\ \mbox{ for }
		s \geqslant 1
	\end{aligned}
	\right.
\]
after implementing the Sobolev embedding.
Similarly,
\begin{equation*}
	\norm{ \left[ J, \eta \right] J^{s - 1/2} v }_{L^2}
	\lesssim
	\lVert \partial_x \eta \rVert _{L^{p_1}}
	\lVert J^{s - 1/2} v \rVert _{L^{p_2}}
	+
	\lVert J \eta \rVert _{L^{p_3}}
	\lVert J^{s - 1/2} v \rVert _{L^{p_4}}
\end{equation*}
follows from the Kato-Ponce inequality.
Now taking $p_1 = p_3 = \frac{1}{1-s}$,
$p_2 = p_4 = \frac{2}{2s-1}$
for $s \in (\frac 12, 1)$
and $p_1 = p_2 = p_3 = p_4 = 4$ for $s \geqslant 1$
one deduces
\[
	\norm{ \left[ J, \eta \right] J^{s - 1/2} v }_{L^2}
	\lesssim
	\norm{ \eta } _{H^{s + 1/2}}
	\left\{
	\begin{aligned}
		\norm{ v } _{H^{1/2}}
		&
		\ \mbox{ for }
		s \in (1/2, 1)
		\\
		\norm{ v } _{H^{s - 1/4}}
		&
		\ \mbox{ for }
		s \geqslant 1
	\end{aligned}
	\right.
\]
after implementing the Sobolev embedding.
Thus applying H\"older's inequality
to the sum $I_1 + I_5$ one obtains
\[
	I_1 + I_5 \lesssim
	\varkappa
	\lVert \partial_x \eta \rVert _{H^{s - 1/2}}
	\norm{ \eta } _{H^{s + 1/2}}
	\left\{
	\begin{aligned}
		\norm{ v } _{H^{1/2}}
		+ \norm{ v } _{L^{\infty}}
		&
		\ \mbox{ for }
		s \in (1/2, 1)
		\\
		\norm{ v } _{H^{s - 1/4}}
		&
		\ \mbox{ for }
		s \geqslant 1
	\end{aligned}
	\right.
	,
\]
and so
\begin{equation*}
	I_1 + I_5 \lesssim
	\varkappa
	\norm{ \partial_x \eta } _{H^{s - 1/2}}
	\left(
		\norm{ \partial_x \eta } _{H^{s - 1/2}}
		+
		\norm{ \eta } _{H^{s - 1/2}}
	\right)
	\lVert v \rVert _{H^s}
	\lesssim
	\norm{ \eta, v } _{ H_{\varkappa}^{s + 1/2} \times H^s }^2
	+ \varkappa
	\norm{ \eta, v } _{ H_{\varkappa}^{s + 1/2} \times H^s }^4
	.
\end{equation*}
Finally, summing Derivative \eqref{dE3}
with the derivative of square of
$H_{\varkappa}^{s + 1/2} \times H^s$-norm
according to Definition \eqref{energy}
we obtain
\[
	\frac{d}{dt}E^s (\eta, v)
	=
	I_1 + \ldots + I_8
	\lesssim
	\norm{ \eta, v } _{ H_{\varkappa}^{s + 1/2} \times H^s }^2
	+ \varkappa
	\norm{ \eta, v } _{ H_{\varkappa}^{s + 1/2} \times H^s }^4
	,
\]
which concludes the proof.
\end{proof}

In the following obvious statement the non-cavitation condition
plays a crucial role.

\begin{lemma}[Coercivity]
	Let $s \geqslant 1/2$ and
	\(
		(\eta,v) \in C\left([0,T];
		H^{s + 1/2}(\mathbb R)
		\times 
		H^s(\mathbb R)\right)
	\)
	be a solution of System \eqref{capillarity_sys} for some $T>0$.
	If in addition $\eta$ satisfies
	the non-cavitation condition then
	\[
		E^s (\eta, v)
		\sim
		\norm{ \eta, v } _{ H_{\varkappa}^{s + 1/2} \times H^s }^2
		.
	\]
\end{lemma}
%

\begin{corollary}[Energy estimate]
\label{energy_estimate_corollary}
	If the conditions of the previous two lemmas are satisfied
	then it holds true
	that
	\[
		\frac{d}{dt}E^s (\eta, v)
		\lesssim
		(1 + \varkappa)
		\left(
			E^s (\eta, v)
			+
			E^s (\eta, v) ^2
		\right)
	\]
	with the implicit constant independent of $\varkappa > 0$.
\end{corollary}

As we shall see below,
the non-cavitation condition
is convenient to work with only
in the case of high regularity $s > 3/2$.
Then the time interval on which the condition
holds true can be easily estimated
through the first equation in \eqref{capillarity_sys}.
Our goal is to study well-posedness in spaces
of low regularity as well.
So in case of $s \leqslant 3/2$ we will have to
impose a stronger condition,
instead of non-cavitation,
namely smallness of the initial data norm, that we can control
in time with the help of the Hamiltonian conservation,
as the following lemma demonstrates.

\begin{lemma}
\label{energy_bound_lemma}
	There exists a constant $H > 0$ independent
	of the surface tension $\varkappa > 0$ such that
	for any $\epsilon \in (0, H]$ if
	a pair
	\(
		u = ( \eta, v )
		\in C\left([0,T];
		H^{s + 1/2}(\mathbb R)
		\times 
		H^s(\mathbb R)\right)
	\),
	having initial condition
	\(
		\lVert u(0) \rVert _{H_{\varkappa}^1 \times H^{1/2}}
		\leqslant \epsilon / 2
	\),
	solves System \eqref{capillarity_sys}
	then
	\(
		\lVert u(t) \rVert _{H_{\varkappa}^1 \times H^{1/2}}
		\leqslant \epsilon
	\)
	for any time $t \in [0,T]$.
\end{lemma}
\begin{proof}
We use a continuity argument.
We simply write
\[
	\lVert u \rVert ^2
	=
	\frac 12 \lVert u(t) \rVert _{H_{\varkappa}^1 \times H^{1/2}}^2
	.
\]
Then
there exists $C > 0$ independent
of $\varkappa > 0$ such that
\[
	\left| \int \eta v^2 \right|
	\leqslant
	\norm{ \eta }_{L^2}
	\norm{ v }_{L^4}^2
	\leqslant
	C \norm{ u }^3
	,
\]
and so
\[
	\lVert u \rVert ^2
	( 1 - C \lVert u \rVert )
	\leqslant
	\mathcal H(u)
	\leqslant
	\lVert u \rVert ^2
	( 1 + C \lVert u \rVert )
	,
\]
where $u = u(t)$ is a solution of \eqref{capillarity_sys}
defined on some interval.
Take $H = (2C)^{-1}$,
any $0 < \epsilon \leqslant H$
and a solution with $u_0 = u(0)$
having $\lVert u_0 \rVert \leqslant \epsilon / 2$.
By continuity
$\lVert u \rVert \leqslant \epsilon$
on some $[0, T_{\epsilon}]$ and so
\[
	\lVert u \rVert
	\leqslant
	\sqrt{2 \mathcal H(u)} = \sqrt{2 \mathcal H(u_0)}
	\leqslant
	\sqrt{\frac {1 + C \epsilon /2}2 } \epsilon < \epsilon
	.
\]
Hence the function $u$ satisfies that
$\lVert u(t) \rVert$ does not reach the level
$\epsilon$ at any time $t$.
%
%
\end{proof}

As a consequence of the lemma we can control
$\lVert \eta \rVert _{L^{\infty}}$
for any $s \geqslant 1/2$ in time,
admitting only small initial data, by the inequality
\[
	\lVert \eta \rVert _{L^{\infty}}
	\lesssim
	\left( 1 + \frac 1{\varkappa} \right)
	\lVert \eta, v \rVert _{H_{\varkappa}^1 \times H^{1/2}}
	,
\]
which
guarantees non-cavitation, in particular.

\section{Uniqueness type estimate}
\label{Uniqueness_type_estimate}
\setcounter{equation}{0}
%
Suppose that
we have two solution pairs $\eta_1$, $v_1$ and $\eta_2$, $v_2$
of System \eqref{capillarity_sys} on some time interval.
Define functions $\theta = \eta_1 - \eta_2$, $w = v_1 - v_2$.
Then $\theta$ and $w$ satisfy the following system
%
\begin{equation}
\label{uniqueness_sys}
\left\{
\begin{aligned}
	\theta_t &=
	- \partial_x w - i \tanh D (\theta v_2 + \eta_1 w)
	, \\
	w_t &=
	- i \tanh D (1 + \varkappa D^2) \theta - i \tanh D ((v_1+v_2)w)/2
	.
\end{aligned}
\right.
\end{equation}
%
We need an a priori estimate similar to one obtained in
the previous section for the difference of solutions.
For this purpose we introduce the difference energy
\begin{equation}
\label{difference_energy}
	E^r(\eta_1, v_1, \eta_2, v_2)
	= \frac {\varkappa}2 \lVert \theta \rVert _{H^{r + 1/2}}^2
	+ \frac 12 \lVert w \rVert _{H^r}^2
	+ \frac 12 \int \eta_1 \left( J^{r - \frac 12 } w \right) ^2
	.
\end{equation}
%
%
\begin{lemma}
\label{difference_energy_lemma}
	Let
	\(
		\eta_1, v_1, \eta_2, v_2 \in C^1 \left( (0, T);
		H^{\infty}(\mathbb{R}) \right)
	\)
	be solutions of System \eqref{capillarity_sys} for some $T>0$
	and $s > 1/2$.
	Their difference is denoted by $(\theta, w)$.
	Let $0 < r \leqslant s - 1/2$.
	Then
	\[
		\frac d{dt} E^r(\eta_1, v_1, \eta_2, v_2)
		\lesssim
		\left(
			1 + \lVert \eta_1 \rVert _{H^{s + 1/2}}^2
			+
			\lVert v_1 \rVert _{H^s}^2
			+
			\lVert v_2 \rVert _{H^s}^2
		\right)
		\left(
			\lVert \theta \rVert _{H^{r + 1/2}}^2
			+		
			\lVert w \rVert _{H^r}^2
		\right)
		,
	\]
	where the implicit constant depends on
	$\varkappa, r, s$.
\end{lemma}
\begin{proof}
We follow the same arguments as in the proof of
Lemma \ref{energy_lemma}.
The derivative of squared norm
\begin{multline*}
	\frac {\varkappa}2 \frac{d}{dt} \lVert \theta \rVert _{H^{r+1/2}}^2
	+
	\frac 12 \frac{d}{dt} \lVert w \rVert _{H^r}^2
	=
	-\varkappa \int \left( J^{r + \frac 12} \theta \right)
	J^{r + \frac 12} \partial_x w
	- i\varkappa \int \left( J^{r + \frac 12} \theta \right)
	J^{r + \frac 12} \tanh D(\theta v_2)
	\\
	- i\varkappa \int \left( J^{r + \frac 12} \theta \right)
	J^{r + \frac 12} \tanh D(\eta_1 w)
	- i \int ( J^r w ) J^r \tanh D \theta
	\\
	- i \varkappa \int ( J^r w ) J^r D^2 \tanh D \theta
	- \frac i2 \int ( J^r w ) J^r \tanh D (v_1 + v_2)w
	\\
	=
	I_1 + \mathcal O
	\left(
		\lVert \theta \rVert _{H^r}
		\lVert w \rVert _{H^r}
		+
		\lVert v_2 \rVert _{H^{r + 1/2}}
		\lVert \theta \rVert _{H^{r + 1/2}}^2
		+
		\lVert \eta_1 \rVert _{H^r}
		\lVert \theta \rVert _{H^r}
		\lVert w \rVert _{H^r}
		+
		\lVert v_1 + v_2 \rVert _{H^s}
		\lVert w \rVert _{H^r}^2
	\right)
	,
\end{multline*}
where
\[
	I_1 =
	i\varkappa \int \left( J^{r - \frac 12} D \theta \right)
	J^{r + \frac 12} (\eta_1 w)
	.
\]
In the case $r \geqslant 1/2$ we have the
commutator estimate
\[
	\left\lVert
		\left[ J^{r + \frac 12}, \eta_1 \right] w
	\right\rVert _{L^2}
	\lesssim
	\lVert \partial_x \eta_1 \rVert _{L^4}
	\left\lVert
		J^{r - \frac 12} w
	\right\rVert _{L^4}
	+
	\left\lVert
		J^{r + \frac 12} \eta_1
	\right\rVert _{L^4}
	\lVert w \rVert _{L^4}
	\lesssim
	\lVert \eta_1 \rVert _{H^{s + 1/2}}
	\lVert w \rVert _{H^r}
	,
\]
and so
\begin{equation}
\label{I1_rest}
	I_1 =
	i\varkappa \int \left( J^{r - \frac 12} D \theta \right)
	\eta_1 J^{r + \frac 12} w
	+ \mathcal O
	\left(
		\lVert \eta_1 \rVert _{H^{s + 1/2}}
		\lVert \theta \rVert _{H^{r + 1/2}}
		\lVert w \rVert _{H^r}
	\right)
	.
\end{equation}
For $r \in (0, 1/2)$ we apply the Leibniz rule
\[
	\left\lVert
		|D|^{r + \frac 12} (\eta_1 w)
		- w |D|^{r + \frac 12} \eta_1
		- \eta_1 |D|^{r + \frac 12} w
	\right\rVert _{L^2}
	\lesssim
	\left\lVert
		|D|^{\sigma_1} \eta_1
	\right\rVert _{L^{p_1}}
	\left\lVert
		|D|^{\sigma_2} w
	\right\rVert _{L^{p_2}}
	\lesssim
	\lVert \eta_1 \rVert _{H^1}
	\lVert w \rVert _{H^r}
	,
\]
where $p_2 > 2$ is such that
$\sigma_2 = r - 1/2 + 1/{p_2} > 0$.
The last estimate is due to
Sobolev's embedding.
Operator
\(
	J^{r + \frac 12 } - |D|^{r + \frac 12 }
\)
is bounded in $L^2$.
Thus
\begin{multline*}
	I_1 =
	i\varkappa \int \left( J^{r - \frac 12} D \theta \right)
	|D|^{r + \frac 12} (\eta_1 w)
	+ \mathcal O
	\left(
		\lVert \eta_1 \rVert _{H^{s + 1/2}}
		\lVert \theta \rVert _{H^{r + 1/2}}
		\lVert w \rVert _{H^r}
	\right)
	\\
	=
	i\varkappa \int \left( J^{r - \frac 12} D \theta \right)
	w |D|^{r + \frac 12} \eta_1
	+	
	i\varkappa \int \left( J^{r - \frac 12} D \theta \right)
	\eta_1 |D|^{r + \frac 12} w
	+ \mathcal O
	\left(
		\lVert \eta_1 \rVert _{H^{s + 1/2}}
		\lVert \theta \rVert _{H^{r + 1/2}}
		\lVert w \rVert _{H^r}
	\right)
	,
\end{multline*}
where the first integral can be estimated by
interpolation in Sobolev spaces.
In the second integral the fractional derivative
\(
	|D|^{r + \frac 12 }
\)
can be approximated by
\(
	J^{r + \frac 12 }
\)
to come again to \eqref{I1_rest} now for $0 < r < 1/2$.

Differentiation of the energy modifier
gives
\begin{multline*}
	\frac 12 \frac {d}{dt}
	\int \eta_1 \left( J^{r - \frac 12 } w \right) ^2
	=
	- i \int \eta_1 \left( J^{r - \frac 12 } w \right)
	J^{r - \frac 12 } \tanh D \theta
	- i \varkappa \int \eta_1 \left( J^{r - \frac 12 } w \right)
	J^{r - \frac 12 } D^2 \tanh D \theta
	\\
	- \frac i2 \int \eta_1 \left( J^{r - \frac 12 } w \right)
	J^{r - \frac 12 } \tanh D (v_1 + v_2)w
	- \frac 12
	\int \partial_x v_1 \left( J^{r - \frac 12 } w \right) ^2
	- \frac i2
	\int \tanh D (\eta_1 v_1) \left( J^{r - \frac 12 } w \right) ^2
	\\
	=
	I_2 + \mathcal O
	\left(
		\lVert \eta_1 \rVert _{H^s}
		\lVert \theta \rVert _{H^{r - 1/2}}
		\lVert w \rVert _{H^{r - 1/2}}
		+
		\left(
			1 + \lVert \eta_1 \rVert _{H^s}
		\right)
		\left(
			\lVert v_1 \rVert _{H^s}
			+			
			\lVert v_2 \rVert _{H^s}
		\right)
		\lVert w \rVert _{H^r}^2
	\right)
	,
\end{multline*}
where
\begin{multline*}
	I_2 =
	- i\varkappa \int \left( J^{r - \frac 12} D \theta \right)
	J (\eta_1 J^{r - \frac 12} w)
	=
	- i\varkappa \int \left( J^{r - \frac 12} D \theta \right)
	\eta_1 J^{r + \frac 12} w
	\\
	+ \lVert \theta \rVert _{H^{r + 1/2}}
	\mathcal O	
	\left(
		\lVert \partial_x \eta_1 \rVert _{L^{p_1}}
		\lVert J^{r - \frac 12} w \rVert _{L^{p_2}}
		+
		\lVert J \eta_1 \rVert _{L^{p_3}}
		\lVert J^{r - \frac 12} w \rVert _{L^{p_4}}
	\right)
	\\
	=
	- i\varkappa \int \left( J^{r - \frac 12} D \theta \right)
	\eta_1 J^{r + \frac 12} w
	+ \mathcal O	
	\left(
		\lVert \eta_1 \rVert _{H^{s + 1/2}}
		\lVert \theta \rVert _{H^{r + 1/2}}
		\lVert w \rVert _{H^r}
	\right)
\end{multline*}
following from the Kato--Ponce inequality
with $p_1 = p_3 = \frac{1}{1-s}$,
$p_2 = p_4 = \frac{2}{2s-1}$
for $s \in \left( \frac 12, 1 \right)$
and $p_1 = p_2 = p_3 = p_4 = 4$ for $s \geqslant 1$.
Summing $I_2$ together with $I_1$ calculated
in \eqref{I1_rest} we conclude the proof.
\end{proof}

\begin{corollary}[Energy estimate for difference]
	If in addition to
	the conditions of the previous lemma we assume
	non-cavitation for $\eta_1$ then
	\[
		\frac d{dt} E^r(\eta_1, v_1, \eta_2, v_2)
		\lesssim
		\left(
			1 + \lVert \eta_1 \rVert _{H^{s + 1/2}}^2
			+
			\lVert v_1 \rVert _{H^s}^2
			+
			\lVert v_2 \rVert _{H^s}^2
		\right)
		E^r(\eta_1, v_1, \eta_2, v_2)
		.
	\]
\end{corollary}
\begin{proof}
Non-cavitation implies coercivity for $E^r$ and the rest is obvious.
\end{proof}

\begin{remark}
	The restriction $s > 1/2$ appeared in the lemma
	and its corollary is inconvenient.
	It comes from the loss of Hamiltonian structure
	of System \eqref{uniqueness_sys}.
	This results in the fact that we can obtain
	only a weak solution in case $s = 1/2$ and
	probably not unique.
\end{remark}

\section{Parabolic regularisation}
\label{Parabolic_regularisation}
\setcounter{equation}{0}

For application of the energy method
we need to do a parabolic regularisation
of the view
\begin{equation}
\label{reg_capillarity_sys}
\left\{
\begin{aligned}
	\eta_t
	+ v_x + i \tanh D (\eta v)
	&= - \varkappa \mu |D|^p \eta
	, \\
	v_t 
	+ i \tanh D (1 + \varkappa D^2) \eta + i \tanh D v^2 / 2
	&= - \varkappa \mu |D|^p v 
\end{aligned}
\right.
\end{equation} 
where $\mu \in (0, 1)$.
We want
to prove solution existence
for \eqref{reg_capillarity_sys} for any given $\mu$,
by the contraction mapping principal
and so $p$ should be big enough.
However, we also do not want to spoil our energy
estimates, and so $p$ should be small enough.
As we shall see below, this bounds us to $p \in (1/2, 1]$.
Here the left number comes from the following lemma.

\begin{lemma}
\label{regularised_estimate_lemma}
	For any $s \geqslant 1/2$, $\mu >0$
	and $p > 1/2$ there exists
	a finite positive bound $C(T)$,
	tending to zero as $T \to 0$,
	such that
	\[
		\int_0^T
		\left \lVert
		e^{ -\mu t |D|^p }
		(f(t)g(t))
		\right \rVert
		_{H^r}
		dt
		\leqslant
		C(T) \norm{f}_{C_TH^r} \norm{g}_{C_TH^s}
	\]
	for any functions $f, g$ defined on $[0, T]$.
	Here either $r = s + 1/2$ or $r = s$.
\end{lemma}

\begin{proof}
In the case $r = s > 1/2$ the statement is obvious
due to boundedness of $\exp( -\mu t |D|^p )$
and the algebraic property
\(
	\norm{fg}_{H^s} \lesssim
	\norm{f}_{H^s} \norm{g}_{H^s}
	.
\)
Hence $C(T) = c_s T$ with some constant $c_s$ depending only
on $s$.

Otherwise we use
\[
	\left \lVert
		e^{ -\mu t |D|^p }
		(fg)
	\right \rVert
	_{H^r}
	\leqslant
	\left \lVert
		\xi \mapsto e^{ -\mu t |\xi|^p }
		\langle \xi \rangle ^{1/2}
	\right \rVert
	_{L^{\infty}}
	\norm{fg}_{H^{r - 1/2}}
\]
where in the case $r = s = 1/2$
by the H\"older inequality we have
\[
	\norm{fg}_{H^{r - 1/2}}
	=
	\norm{fg}_{L^2}
	\leqslant
	\norm{f}_{L^4} \norm{g}_{L^4}
	\lesssim
	\norm{f}_{H^{1/4}} \norm{g}_{H^{1/4}}
	\lesssim
	\norm{f}_{H^s} \norm{g}_{H^s}
\]
and in the case $r = s + 1/2$ we obviously have
\[
	\norm{fg}_{H^{r - 1/2}}
	\lesssim
	\norm{f}_{H^r} \norm{g}_{H^s}
	.
\]
It is left to check that the $L^{\infty}$-norm above
is locally integrable.
Indeed, we can estimate the
function at $\xi \in [0, 1]$
and at $\xi \geqslant 1$ separately
\[
	e^{ -\mu t |\xi|^p }
	\langle \xi \rangle ^{1/2}
	\leqslant
	\max \left \{
		2^{1/4},
		\sup _{\xi \geqslant 1}
		2^{1/4} \xi^{\frac 1{2p}}
		e^{ -\mu t |\xi| }
	\right \}
	\leqslant
	2^{1/4} \max \left \{
		1,
		( 2p e \mu t )^{ -\frac 1{2p} }
	\right \}
\]
that is an integrable function with respect to time
over any bounded interval for $p > 1/2$.
The integral of this function over $[0, T]$ defines the bound $C(T)$.
\end{proof}

With Lemma \ref{regularised_estimate_lemma} in hand
we can prove the local well-posedness in
$H^{s + 1/2} (\mathbb R) \times H^s (\mathbb R)$
with $s \geqslant 1/2$ for System \eqref{reg_capillarity_sys}
by the fixed-point argument.
Diagonalization has the matrix form
\begin{equation}
\label{semigroup}
	\mathcal S(t) =
	\exp(-\varkappa \mu t |D|^p)
	\mathcal K
	\begin{pmatrix}
		\exp(-itK_{\varkappa}D) & 0
		\\
		0 & \exp(itK_{\varkappa}D)
	\end{pmatrix}
	\mathcal K^{-1}
	,
\end{equation}
where
\[
	\mathcal K
	=
	\frac 1{\sqrt 2}
	\begin{pmatrix}
		1 & 1
		\\
		K_{\varkappa} & -K_{\varkappa}
	\end{pmatrix}
	, \quad
	\mathcal K^{-1}
	=
	\frac 1{\sqrt 2}
	\begin{pmatrix}
		1 & K_{\varkappa}^{-1}
		\\
		1 & -K_{\varkappa}^{-1}
	\end{pmatrix}
\]
with $K_{\varkappa}$
defined by \eqref{K_kappa}.
For any fixed
\(
	u_0 = (\eta_0, v_0)^T \in X^s
	= H^{s + 1/2} (\mathbb R) \times H^s (\mathbb R)
\)
the function $\mathcal S(t)u_0$ solves the linear
initial-value problem associated with
\eqref{reg_capillarity_sys}.
Let $X^s_T = C([0, T]; X^s)$ and
regard a mapping $\mathcal A : X^s_T \to X^s_T$
defined by
\begin{equation}
\label{contraction_mapping}
	\mathcal A(\eta, v; u_0)(t) = \mathcal S(t)u_0
	+ \int_0^t \mathcal S(t - t') (-i\tanh D)
	\begin{pmatrix}
		\eta v
		\\
		v^2 / 2
	\end{pmatrix}
	(t')dt'
	.
\end{equation}
Then the Cauchy problem for System \eqref{reg_capillarity_sys}
with the initial data $u_0$ may be rewritten equivalently
as an equation in $X^s_T$ of the form
\begin{equation}
\label{u_is_Au}
	 u = \mathcal A(u; u_0)
\end{equation}
where $u = (\eta, v)^T \in X^s_T$.
\begin{lemma}
\label{regularised_local_existence_lemma}
	Let $s \geqslant 1/2$, $p > 1/2$,
	$\mu \in (0, 1)$ and
	$u_0 = (\eta_0, v_0)^T \in X^s$.
	Then there is
	$T = T( s, p, \varkappa, \mu, \lVert u_0 \rVert_{X^s} ) > 0$,
	decreasing to zero with increase of the norm of $u_0$,
	such that there exists a unique solution
	$u = (\eta, v)^T \in X^s_T$ of Problem \eqref{u_is_Au}.

	Moreover, for any $R > 0$ there exists
	a $T = T( s, p, \varkappa, \mu, R) > 0$
	such that the flow map associated with Equation \eqref{u_is_Au}
	is a real analytic mapping of the open ball
	$B_R(0) \subset X^s$ to $X^s_T$.
\end{lemma}
\begin{proof}
We need to show that the restriction of $\mathcal A$
on some closed ball $B_M$
with the center at point $\mathcal S(t)u_0$
is a contraction mapping.
Note that
\(
	\norm{ \mathcal S(t)u } _{X^s}
	\lesssim
	\norm{ \exp(- \varkappa \mu t |D|^p) u } _{X^s}
	.
\)
Hence by Lemma \ref{regularised_estimate_lemma}
for any $T, M > 0$ and $u, u_1, u_2 \in B_M \subset X^s_T$ hold
\[
	\lVert \mathcal A(u) - \mathcal S(t)u_0 \rVert _{X^s_T}
	\leqslant
	C(T) \lVert u \rVert _{X^s_T}^2
	\leqslant
	C(T) ( M + \lVert u_0 \rVert _{X^s} )^2
	,
\]
\[
	\lVert \mathcal A(u_1) - \mathcal A(u_2) \rVert _{X^s_T}
	\leqslant
	C(T) \lVert u_1 - u_2 \rVert _{X^s_T}
	(
		\lVert u_1 \rVert _{X^s_T}
		+
		\lVert u_2 \rVert _{X^s_T}
	)
	\leqslant
	2 C(T) ( M + \lVert u_0 \rVert _{X^s} )
	\lVert u_1 - u_2 \rVert _{X^s_T}
	,
\]
and so taking $M = \lVert u_0 \rVert _{X^s}$
one can find a $T > 0$ such that $\mathcal A$ will be a contraction in
the closed ball $B_M$.
The first statement of the lemma follows from the
contraction mapping principle.
Smoothness of the flow map can be proved in the same spirit
applying the implicit function theorem instead,
and so we omit it.
Some details can be found in \cite{Dinvay_Tesfahun}.
\end{proof}

By a standard argumentation, see for example \cite{Iorio_Iorio},
one can show that if
$u = (\eta, v)^T \in X^s_T$ is the solution of Problem \eqref{u_is_Au}
then
\(
	u \in C^1 \left( (0, T);
	H^{s - 1}(\mathbb{R}) \times H^{s - 3/2}(\mathbb{R}) \right)
\)
and it solves the regularised system \eqref{reg_capillarity_sys}
as well with the initial data $u_0 \in X^s$.
Clearly, in order to be able to use the following energy
and a priori estimates,
one has to pick up a smooth initial data.
The justification is discussed briefly in
Section \ref{Proof_of_Theorem}.

\begin{lemma}
\label{regularised_energy_lemma}
	Suppose $s \geqslant 1/2$.
	Then there exists $C_s > 0$ such that
	for any $\varkappa > 0$ and
	any functions	
	\(
		\eta, v \in C^1 \left( (0, T);
		H^{\infty}(\mathbb{R}) \right)
	\)
	solving System \eqref{reg_capillarity_sys}
	with $\mu \in (0, 1)$ and $p \in (1/2, 1]$
	it holds that
	\[
		\frac{d}{dt}E^s (\eta, v)
		\leqslant
		C_s (1 + \varkappa)
		\left(
			\norm{ \eta, v } _{ H_{\varkappa}^{s + 1/2} \times H^s }^2
			+
			\norm{ \eta, v } _{ H_{\varkappa}^{s + 1/2} \times H^s }^4
		\right).
	\]
\end{lemma}
In other words, the parabolic regularisation
\eqref{reg_capillarity_sys} does not spoil
the energy estimate.
Note that $C_s > 0$ does not depend on
$\varkappa, \mu, p$.
\begin{proof}
Following the proof of Lemma \ref{energy_lemma}
one arrives at
\begin{equation}
\label{dE}
	\frac{d}{dt}E^s (\eta, v)
	=
	\widetilde I_1 + \widetilde I_2
	+ I_1 + \ldots + I_8
	,
\end{equation}
where
\[
	\widetilde I_1 =
	- \varkappa^2 \mu
	\left \lVert
		\partial_x |D|^{p/2} \eta
	\right \rVert
	_{H^{s - 1/2}}^2
	- \varkappa \mu
	\left \lVert
		|D|^{p/2} \eta
	\right \rVert
	_{H^{s - 1/2}}^2
	- \varkappa \mu
	\left \lVert
		K^{-1} |D|^{p/2} v
	\right \rVert
	_{H^{s - 1/2}}^2
	\leqslant 0
	,
\]
\[
	\widetilde I_2 =
	- \frac{\varkappa \mu}2
	\int \left( J^{s - 1/2} v \right) ^2
	|D|^p \eta
	- \varkappa \mu
	\int \eta \left( J^{s - 1/2} v \right)
	J^{s - 1/2} |D|^p v
	\lesssim
	\varkappa
	\lVert \eta \rVert _{H^{s + 1/2}}
	\lVert v \rVert _{H^s}^2
\]
for $p \leqslant 1$
and the rest integrals $I_1, \ldots, I_8$
are the same as in Lemma \ref{energy_lemma}.
\end{proof}

As was noticed at the end of Section \ref{Modified_energy},
one has to make sure that the modified energy is coercive.
An effective way to do it at the low level of regularity
is to control $\norm{\eta} _{L^{\infty}}$
via the energy conservation.
One can get the same controllability for the regularised
problem via the energy dissipation due to the following result.

\begin{lemma}
\label{energy_dissipation_lemma}
	Suppose
	\(
		\eta, v \in C^1 \left( (0, T);
		H^{\infty}(\mathbb{R}) \right)
	\)
	solve
	System \eqref{reg_capillarity_sys}
	with $\varkappa > 0$, $\mu \in (0, 1)$ and $p \in (1/2, 1]$.
	Then there exists $\delta (p) > 0$ independent of
	the viscosity $\mu$ and the capillarity $\varkappa$ such that
	$\mathcal H (\eta, v)$ is a non-increasing
	function of time $t$ provided
	\(
		\norm{\eta(t)} _{L^2} + \norm{v(t)} _{H^{1/2}}
		\leqslant \delta
	\)
	holds for any moment $t$.
\end{lemma}
\begin{proof}
Hamiltonian \eqref{Hamiltonian} has the derivative
\[
	\frac 1{\varkappa \mu} \frac{d}{dt} \mathcal H (\eta, v)
	=
	- \norm{ \eta } _{\dot{H}^{p/2}}^2
	- \varkappa \norm{ \eta } _{\dot{H}^{p/2 + 1}}^2
	- \norm{ K^{-1} v } _{\dot{H}^{p/2}}^2
	- I_1
	- I_2
	,
\]
where the rest integrals
\[
	I_1 = \int \eta v |D|^pv
	, \quad
	I_2 = \frac 12 \int v^2 |D|^p\eta
\]
are of no definite sign.
One has to check that $I_1$, $I_2$ are absorbed
by the first and third norms.

Firstly, we rewrite $I_1$ in the form
\begin{multline*}
	I_1
	=
	\int |D|^{p/2}(\eta v) |D|^{p/2}v
	=
	\int
	\left(
		|D|^{p/2}(\eta v) - v |D|^{p/2}\eta - \eta |D|^{p/2}v
	\right)
	|D|^{p/2}v
	\\
	+
	\int \left( |D|^{p/2}\eta \right) v |D|^{p/2}v
	+
	\int \eta \left( |D|^{p/2}v \right)^2
	.
\end{multline*}
Applying the H\"older inequality and the fractional Leibniz rule
\eqref{LeibnizRule}
for $|D|^{p/2}$ with $L^2$-norm to the first integral,
Lemma \ref{Selberg_Tesfahun_lemma} to the second integral
and the H\"older inequality to the third integral,
one obtains
\begin{multline*}
	|I_1|
	\lesssim
	\left \lVert
		|D|^{p/2 - 1/4}\eta
	\right \rVert _{L^4}
	\left \lVert
		|D|^{1/4}v
	\right \rVert _{L^4}
	\left \lVert
		|D|^{p/2}v
	\right \rVert _{L^2}
	\\
	+
	\left \lVert
		|D|^{p/2}\eta
	\right \rVert _{L^2}
	\norm{v} _{H^{1/2}}
	\left \lVert
		|D|^{p/2}v
	\right \rVert _{H^{1/2}}
	+
	\norm{ \eta } _{L^2}
	\left \lVert
		|D|^{p/2}v
	\right \rVert _{L^4}^2
	.
\end{multline*}
Using the Sobolev embedding
\(
	\dot H^{1/4} \hookrightarrow L^4
	,
\)
one finally obtains
\[
	|I_1|
	\lesssim
	\left \lVert
		|D|^{p/2}\eta
	\right \rVert _{L^2}
	\norm{v} _{H^{1/2}}
	\left \lVert
		|D|^{p/2}v
	\right \rVert _{H^{1/2}}
	+
	\norm{ \eta } _{L^2}
	\left \lVert
		|D|^{p/2}v
	\right \rVert _{H^{1/2}}^2
	.
\]

The second integral $I_2$ can be treated by the H\"older inequality
as follows
\[
	|I_2| = \frac 12
	\int \left( |D|^{p/2}v^2 \right) |D|^{p/2}\eta
	\leqslant
	\frac 12
	\norm{ |D|^{p/2}v^2 } _{L^2}
	\norm{ |D|^{p/2}\eta } _{L^2}
	.
\]
Here the first norm is estimated with the help
of the Leibniz rule in the way
\[
	\norm{ |D|^{p/2}v^2 } _{L^2}
	\lesssim
	\norm{ v|D|^{p/2}v } _{L^2}
	+
	\norm{ |D|^{p/2}v } _{L^4}^2
	\lesssim
	\norm{ v } _{H^{1/2}} \norm{ |D|^{p/2}v } _{H^{1/2}}
	+
	\norm{ |D|^{p/2 + 1/4}v } _{L^2}^2
	,
\]
where we have used Lemma \ref{Selberg_Tesfahun_lemma}
and the embedding 
\(
	\dot H^{1/4} \hookrightarrow L^4
	.
\)
Thus
\[
	|I_2|
	\lesssim
	\norm{ v } _{H^{1/2}} \norm{ |D|^{p/2}v } _{H^{1/2}}
	\norm{ |D|^{p/2}\eta } _{L^2}
	.
\]
Eventually we obtain
\[
	|I_1| + |I_2|
	\lesssim
	\left(
		\norm{ \eta } _{\dot{H}^{p/2}}^2
		+ \norm{ K^{-1} v } _{\dot{H}^{p/2}}^2
	\right)
	\max \left\{ \norm{\eta} _{L^2}, \norm{v} _{H^{1/2}} \right\}
	,
\]
that concludes the proof.
Note that the implicit constant here does not depend on $\varkappa$.

\end{proof}

As a simple corollary with the proof similar to that of
Lemma \ref{energy_bound_lemma} one obtains the following.

\begin{corollary}
\label{reg_energy_bound_corollary}
	There exists a constant $\delta > 0$, depending only
	on the parabolic regularization power $p$,
	such that
	if a pair
	\(
		u = ( \eta, v ) \in C^1 \left( (0, T);
		( H^{\infty}(\mathbb{R}) )^2 \right)
	\),
	having initial condition
	\(
		\lVert u_0 \rVert _{H_{\varkappa}^1 \times H^{1/2}}
		\leqslant \delta / 2
	\),
	solves System \eqref{reg_capillarity_sys}
	then
	\(
		\lVert u(t) \rVert _{H_{\varkappa}^1 \times H^{1/2}}
		\leqslant \delta
	\)
	for any time $t$.
\end{corollary}

The dependence of $\delta$ on the parabolic regularisation power $p$
is unimportant since below we stick only
to the case $p = 1$.
%
%
\section{A priori estimate}
\label{A_priori_estimate}
\setcounter{equation}{0}
%
%
We have an a priori global bound for solutions of both systems
\eqref{capillarity_sys} and \eqref{reg_capillarity_sys}
in $H^1(\mathbb R) \times H^{1/2}(\mathbb R)$
due to Lemma \ref{energy_bound_lemma}	
and Corollary \ref{reg_energy_bound_corollary}, respectively.
Our aim is it to obtain estimates in
$H^{s + 1/2}(\mathbb R) \times H^s(\mathbb R)$
with $s > 1/2$.

\begin{lemma}
[A priori estimate]
\label{a_priori_estimate_lemma}
	Suppose $s > 1/2$ and $\varkappa > 0$.
	Let
	\[
		(\eta, v) \in C
		\left(
			[0, T^{\star});
			H^{s + 1/2}(\mathbb R) \times H^s(\mathbb R)
		\right)
		\cap C^1 \left( (0, T^{\star});
		( H^{\infty}(\mathbb{R}) )^2 \right)
	\]
	be a solution of System \eqref{capillarity_sys}
	(or of the regularised system \eqref{reg_capillarity_sys}
	with $\mu \in (0, 1)$ and $p=1$)
	defined on its maximal time interval of existence and
	satisfying the blow-up alternative
	\begin{equation}
	\label{blow_up_alternative}
		T^{\star} < +\infty
		\mbox{ implies }
		\lim _{t \to T^{\star}}
		\norm{\eta(t), v(t)}
		_{H_{\varkappa}^{s + 1/2} \times H^s} = + \infty
		.
	\end{equation}
	Suppose that its initial data \eqref{capillarity_data}
	either satisfies the non-cavitation condition
	for $s > 3/2$ or has small enough
	$H_{\varkappa}^1 \times H^{1/2}$-norm
	for $s \leqslant 3/2$.
	Then there exists
	\(
		T_0
		< T^{\star}
	\)
	such that
	\begin{equation}
	\label{a_priori_estimate}
		\sup _{ t \in [0, T_0] }
		\norm{\eta(t), v(t)} _{H_{\varkappa}^{s + 1/2} \times H^s}
		\leqslant
		C \norm{\eta_0, v_0} _{H_{\varkappa}^{s + 1/2} \times H^s}
	\end{equation}
	for some $C > 0$ independent of $\varkappa, \mu$.
	The time of existence $T_0$ is a non-increasing
	function of the surface tension $\varkappa$ and
	of the initial data norm
	\(
		\norm{ \eta_0, v_0 } _{ H_{\varkappa}^{s + 1/2} \times H^s }
		.
	\)
\end{lemma}

\begin{proof}
We closely follow the arguments in \cite{Kalisch_Pilod}
since we have essentially the same energy estimates.
The main difference lies in the control of coercivity
of the modified energy \eqref{energy}
for small $s$.
Let $h_0, H_0$ define non-cavitation of $\eta_0$
according to Definition \ref{noncavitation_definition}.
Regard $h = h_0/2$ and $H = H_0 + h_0/2$.
If the wave $\eta$ satisfies the non-cavitation
condition associated with $h, H$ then
there exist positive constants $c_0(h)$, $C_0(H)$
such that
	\[
		c_0 \norm{\eta, v} _{H_{\varkappa}^{s + 1/2} \times H^s}^2
		\leqslant		
		E^s (\eta, v)
		\leqslant		
		C_0 \norm{\eta, v} _{H_{\varkappa}^{s + 1/2} \times H^s}^2
	\]
by coercivity of the energy.
These constants depend only on $h_0, H_0$.
They are used to
define the time set
\[
	\mathcal T =
	\left \{
		T \in (0, T^{\star})
		\ : \
		\sup _{ t \in [0, T] }
		\norm{\eta(t), v(t)} _{H_{\varkappa}^{s + 1/2} \times H^s}
		\leqslant
		3 \sqrt{C_0/c_0}
		\norm{\eta_0, v_0} _{H_{\varkappa}^{s + 1/2} \times H^s}
	\right \}
\]
that is non-empty and closed in $(0, T^{\star})$
by the solution continuity.
Moreover, for $\widetilde T = \sup \mathcal T$ we have either
$\widetilde T < T^{\star}$
and so $\widetilde T \in \mathcal T$
or $\widetilde T = T^{\star} = +\infty$
by the blow-up alternative \eqref{blow_up_alternative}.
Introduce $T_0 = \min \{ T_1, T_2\}$ with
\[
	T_1 = \frac 1{C_1 (1 + \varkappa)} \log
	\left(
		1 + \frac 1
		{
			1 + C_1 (1 + \varkappa) C_0
			\norm{\eta_0, v_0} _{H_{\varkappa}^{s + 1/2} \times H^s}^2
		}
	\right)
	,
\]
\[
	T_2 =
	\left\{
	\begin{aligned}
		\frac {h_0}
		{
			C_2
			\left(
				\norm{\eta_0, v_0}
				_{H_{\varkappa}^{s + 1/2} \times H^s}
				+
				\norm{\eta_0, v_0}
				_{H_{\varkappa}^{s + 1/2} \times H^s}^2
			\right)		
		}
		&
		\ \mbox{ for }
		s > 3/2
		\\
		1 
		&
		\ \mbox{ otherwise }
	\end{aligned}
	\right.
	,
\]
where $C_1, C_2$ are two big positive constants to be fixed
below in the proof.
The idea is to show that these constants can be chosen, independently
on the initial data, in such a way that $T_0 \in \mathcal T$
or equivalently $T_0 \leqslant \widetilde T$.

Assume the opposite $\widetilde T < T_0$.
Firstly, we will check that the non-cavitation condition
holds on $[0, \widetilde T]$.
Indeed, in the low regularity case $ s \in (1/2, 3/2]$
it is assumed smallness of the initial data and so
$H_{\varkappa}^1 \times H^{1/2}$-norm of the solution
stays small with time evolution by
Lemma \ref{energy_bound_lemma} and
Corollary \ref{reg_energy_bound_corollary}.
In particular, the wave satisfies the non-cavitation condition.
For $s > 3/2$ one can estimate $\eta$ using the first equation
in System \eqref{capillarity_sys}
(or in System \eqref{reg_capillarity_sys}) as follows
\[
	\eta(x, t) = \eta_0(x) + \int_0^t \partial_t \eta(x, t')dt'
	,
\]
where
\begin{multline*}
	\norm{\partial_t \eta} _{L^{\infty}}
	\leqslant
	\norm{\partial_x v} _{L^{\infty}}
	+
	\norm{\tanh D (\eta v)} _{L^{\infty}}
	+
	\varkappa \mu \norm{ |D| \eta } _{L^{\infty}}
	\\	
	\lesssim
	\norm{ \partial_x v } _{H^{s-1}}
	+ \norm{ \eta } _{H^{s - 1}} \norm{ v } _{H^{s - 1}}
	+ \varkappa \norm{ \partial_x \eta } _{H^{s - 1}}
\end{multline*}
with the implicit constant independent on $\mu \in (0, 1)$, obviously.
Hence
\[
	\norm{\partial_t \eta} _{L^{\infty}}
	\lesssim
	\norm{\eta_0, v_0} _{H_{\varkappa}^{s + 1/2} \times H^s}
	+ \norm{\eta_0, v_0} _{H_{\varkappa}^{s + 1/2} \times H^s}^2
\]
uniformly on $(0, \widetilde T] \subset \mathcal T$.
Thus we have
\[
	\left \lVert
		\int_0^t \partial_t \eta(x, t')dt'
	\right \rVert _{L^{\infty}}
	\leqslant
	\widetilde T \sup _{ t \in (0, \widetilde T] }
	\norm{\partial_t \eta(t)} _{L^{\infty}}
	\leqslant
	\frac {h_0}2
\]
for big enough $C_2$ since $\widetilde T < T_2$.
As a result the non-cavitation
\[
	h - 1 =
	h_0/2 - 1 \leqslant \eta \leqslant H_0 + h_0/2
	= H
\]
holds on $\mathbb R \times [0, \widetilde T ]$.
Without loss of generality one can assume
that for $s \leqslant 3/2$ the non-cavitation of
$\eta$ is governed by the same constants $h, H$.

Let $E(t) = E^s(\eta, v)(t)$ be the energy defined by \eqref{energy}
and $E_0 = E(0)$.
For System \eqref{capillarity_sys}
(or for System \eqref{reg_capillarity_sys}) we
have the a priori energy estimate
given in its differential form 
by Corollary \ref{energy_estimate_corollary}.
It can be rewritten in the form
\[
	\left(
		\frac E{1 + E}
	\right)'
	\leqslant c (1 + \varkappa) \frac E{1 + E}
	.
\]
A straightforward use of Gr\"onwal's inequality gives
\[
	E(t)
	\left(
		1 - \frac{E_0}{1 + E_0} e^{c (1 + \varkappa) t}
	\right)
	\leqslant \frac{E_0}{1 + E_0} e^{c (1 + \varkappa) t}
\]
for any $t \in [0, \widetilde T ]$
with $c$ depending only on $h$, $s$.
Note that
\[
	e^{c (1 + \varkappa) t} \leqslant
	1 + \frac 1
	{
		1 + C_1 (1 + \varkappa) E_0
	}
\]
for any $C_1 \geqslant c$ and
$0 \leqslant t \leqslant \widetilde T < T_1$.
In particular,
\[
	\frac{E_0}{1 + E_0} e^{c (1 + \varkappa) t} \leqslant
	\frac{ ( C_1 (1 + \varkappa) )^{-1} + E_0}{1 + E_0} < 1
\]
if in addition $C_1 (1 + \varkappa) > 1$.
Thus
\[
	E(t) \leqslant \frac 1
	{
		\left(
			\frac{E_0}{1 + E_0} e^{c (1 + \varkappa) t}
		\right) ^{-1} - 1
	}
	\leqslant
	E_0 \frac
	{ 2 + C_1 (1 + \varkappa) E_0 }
	{ 1 + (C_1 (1 + \varkappa) - 1)E_0 }
	\leqslant
	2E_0
\]
if in addition $C_1 (1 + \varkappa) \geqslant 2$.
As a result setting $C_1 = \max \{ 2, c \}$ we have
\[
	\norm{\eta(t), v(t)} _{H_{\varkappa}^{s + 1/2} \times H^s}
	\leqslant
	\sqrt{2C_0/c_0}
	\norm{\eta_0, v_0} _{H_{\varkappa}^{s + 1/2} \times H^s}
\]
for all $t \in [0, \widetilde T ]$.
Taking into account $\widetilde T < T^{\star}$
and continuity of the solution one can find
$\widetilde T < T' < T^{\star}, T_0$
such that on $[0, T']$ holds
\[
	\norm{\eta(t), v(t)} _{H_{\varkappa}^{s + 1/2} \times H^s}
	\leqslant
	2 \sqrt{C_0/c_0}
	\norm{\eta_0, v_0} _{H_{\varkappa}^{s + 1/2} \times H^s}
\]
which contradicts the definition of $\widetilde T$.
Therefore, we showed that $T_0 \leqslant \widetilde T$
concluding the main part of the proof.

It is left to specify the dependence of $T_0$
on the initial data and the surface tension.
From its definition
one can see that
$T_0$ is non-increasing as
a function of the initial data norm
for each $\varkappa > 0$ fixed.
One can also see straightaway that
$T_0$ is non-increasing as
a function of $\varkappa$
for each $(\eta_0, v_0)$ and $s > 3/2$ fixed.
To the same conclusion one can easily come
in the case $s \leqslant 3/2$,
taking into account that 
the smallness assumption imposed on
the initial data norm is $\varkappa$-independent
according to Corollary \ref{reg_energy_bound_corollary}.
\end{proof}

Studying the low surface tension regime
in the last section, we will appeal to the following remark.

\begin{remark}
\label{uniform_a_priori_estimate_remark}
Suppose that $\varkappa \in (0, K]$.
Then
\(
	\norm{u_0} _{H_{\varkappa}^{s + 1/2} \times H^s}
	\leqslant	
	\norm{u_0} _{H_{K}^{s + 1/2} \times H^s}
\)
and
\(
	T_0(K) \leqslant T_0(\varkappa)
\).
Thus $T_0(K)$ and the smallness parameter of
$H_{K}^1 \times H^{1/2}$-norm of the initial data
$u_0 = (\eta_0, v_0)$
can serve as bounds independent of $\varkappa$
instead of the corresponding bounds given
in the statement of the lemma.
\end{remark}

\begin{lemma}
\label{persistence_of_regularity_lemma}
	Suppose $s > 1/2$, $\varkappa > 0$ and
	functions
	\(
		\eta, v \in C^1 \left( (0, T);
		H^{\infty}(\mathbb{R}) \right)
	\)
	solve System \eqref{capillarity_sys}
	(or the regularised system \eqref{reg_capillarity_sys}
	with $\mu \in (0, 1)$ and $p=1$).
	Then if $s < 1$ the following holds true
	\[
		\frac{d}{dt} E^s (\eta, v)
		\leqslant
		C_s (1 + \varkappa)
		\left(
			1 + \norm{ v } _{L^{\infty}}
			+ \norm{\eta, v} _{H_{\varkappa}^1 \times H^{1/2}}^2
		\right)
		\norm{\eta, v} _{H_{\varkappa}^{s + 1/2} \times H^s}^2
		,
	\]
	and if $s \geqslant 1$ then
	\[
		\frac{d}{dt} E^s (\eta, v)
		\leqslant
		C_s (1 + \varkappa)
		\left(
			1 + \norm{\eta, v}
			_{H_{\varkappa}^{s + 1/4} \times H^{s - 1/4}}^2
		\right)
		\norm{\eta, v} _{H_{\varkappa}^{s + 1/2} \times H^s}^2
		.
	\]
	Moreover, the constant $C_s$ does not depend on $\varkappa$, $\mu$.
\end{lemma}
\begin{proof}
The estimates obtained while proving
Lemmas \ref{energy_lemma}, \ref{regularised_energy_lemma}
need to be refined for $s > 1/2$ as follows.
We stick to the notations used in the corresponding proofs.
Recall Identity \eqref{dE} and note that
$\widetilde I_1$, $I_1 + I_5$, $I_3$, $I_6$, $I_8$
need not to be refined.
So it is left to reconsider only the integrals
$\widetilde I_2$, $I_2$, $I_4$, $I_7$.
Note that by Lemma \ref{Selberg_Tesfahun_lemma} we have
\begin{multline*}
	\widetilde I_2
	\lesssim
	\varkappa
	\norm{ J^{s - 1/2} v } _{L^2}
	\norm{ J^{s - 1/2} v } _{H^{1/2}}
	\norm{ |D| \eta } _{H^{s - 1/2}}
	+
	\varkappa
	\norm{ \eta } _{H^s}
	\norm{ J^{s - 1/2} v } _{H^{1/2}}
	\norm{ J^{s - 1/2}|D| v } _{H^{-1/2}}
	\\
	\lesssim
	(1 + \varkappa)
	\norm{\eta, v} _{H_{\varkappa}^s \times H^{s - 1/2}}
	\norm{\eta, v} _{H_{\varkappa}^{s + 1/2} \times H^s}^2
	.
\end{multline*}
In order to refine $I_7$ we need to estimate
\[
	\left \lVert
		\left( \sgn D |D|^{1/2} v \right)
		J^{s - 1/2 } v 
	\right \rVert _{L^2}
	\lesssim
	\left \lVert
		\left( \sgn D |D|^{1/2} v \right)
	\right \rVert _{L^{p_1}}
	\left \lVert
		J^{s - 1/2 } v 
	\right \rVert _{L^{p_2}}
\]
following from H\"older's inequality
with
$p_1(s) = \frac{1}{1-s}$, $p_2(s) = \frac{2}{2s-1}$
for $s \in (\frac 12, 1)$
and $p_1 = p_2 = 4$ in case $s \geqslant 1$.
Implementing the Sobolev embedding and gathering
the rest of $I_7$ one obtains
\[
	I_7 \lesssim
	\norm{ v } _{H^s}^2
	\left\{
	\begin{aligned}
		\norm{ v } _{H^{1/2}}
		&
		\ \mbox{ for }
		s \in (1/2, 1)
		\\
		\norm{ v } _{H^{s - 1/4}}
		&
		\ \mbox{ for }
		s \geqslant 1
	\end{aligned}
	\right.
	.
\]
It turns out that $I_2$ and $I_4$ should be estimated together
in order to make sure that the constant $C_s$ in the statement
does not depend on $\varkappa$.
Summing $I_2$ and $I_4$ one obtains
\[
	I_2 + I_4 =
	i \int \left( J^{s - 1/2} \tanh D \eta \right)
	\left(
		J^{s - 1/2} (\eta v)
		- \eta J^{s - 1/2 } v
	\right)
	.
\]
Firstly, we regard the case $s > 3/2$ and appeal to
the Kato-Ponce inequality \eqref{Kato_Ponce} to estimate
the commutator above as
\[
	\norm{ \left[ J^{s - 1/2}, \eta \right]v }_{L^2}
	\lesssim
	\lVert \partial_x \eta \rVert _{L^{p_1}}
	\lVert J^{s - 3/2} v \rVert _{L^{p_2}}
	+
	\lVert J^{s - 1/2} \eta \rVert _{L^2}
	\lVert v \rVert _{L^{\infty}}
	.
\]
Taking $p_1(s) = \frac{1}{2-s}$, $p_2(s) = \frac{2}{2s-3}$
for $s \in \left( \frac 32, 2 \right)$
and $p_1 = p_2 = 4$ in case $s \geqslant 2$
one deduces
\[
	I_2 + I_4 \lesssim
	\norm{ \eta } _{H^{s - 1/2}}^2
	\left\{
	\begin{aligned}
		\norm{ v } _{H^{1/2}}
		+ \norm{ v } _{L^{\infty}}
		&
		\ \mbox{ for }
		s \in (3/2, 2)
		\\
		\norm{ v } _{H^{s - 5/4}}
		+ \norm{ v } _{L^{\infty}}
		&
		\ \mbox{ for }
		s \geqslant 2
	\end{aligned}
	\right.
	.
\]
Secondly, in the case $s = 3/2$
the commutator is estimated straightforwardly as
\[
	\norm{ \left[ J^{s - 1/2}, \eta \right]v }_{L^2}
	\lesssim
	\norm{ \eta } _{H^1}
	\norm{ v } _{H^1}
	+
	\norm{ \eta } _{L^{\infty}}
	\norm{ v } _{H^1}
	,
\]
and so
\[
	I_2 + I_4 \lesssim
	\norm{ \eta } _{H^{s - 1/2}}^2
	\norm{ v } _{H^{s - 1/2}}
	.
\]
Finally, regarding the left case $s \in (1/2, 3/2)$
we firstly approximate the Bessel potential $J^{s - 1/2}$
by the Riesz potential $|D|^{s - 1/2}$ in the commutator as
\[
	\norm{
		\left(
			J^{s - 1/2} - |D|^{s - 1/2}
		\right)
		(\eta v)
		- \eta
		\left(
			J^{s - 1/2} - |D|^{s - 1/2}
		\right)
		v
	} _{L^2}
	\lesssim
	\norm{ \eta } _{L^2} \norm{ v } _{L^{\infty}}
	+
	\norm{ \eta } _{L^2} \norm{ v } _{L^2}
	,
\]
and then appealing to the Leibniz rule \eqref{LeibnizRule}
we obtain
\[
	\norm{
		|D|^{s - 1/2} (\eta v)
		- \eta |D|^{s - 1/2} v
	} _{L^2}
	\lesssim
	\norm{ |D|^{s - 1/2} \eta } _{L^2} \norm{ v } _{L^{\infty}}
	.
\]
Hence for $s \in (1/2, 3/2)$ the sum of $I_2$ and $I_4$ is estimated as
\[
	I_2 + I_4 \lesssim
	\norm{ \eta } _{H^{s - 1/2}}^2
	\left(
		\norm{ v } _{L^{\infty}}
		+
		\norm{ v } _{L^2}
	\right)
	.
\]
Thus gathering all the parts one obtains
\[
	\widetilde I_1 + \widetilde I_2
	+ I_1 + \ldots + I_8
	\lesssim
	(1 + \varkappa)
	\norm{\eta, v} _{H_{\varkappa}^{s + 1/2} \times H^s}^2
	\left\{
	\begin{aligned}
		1 + \norm{ v } _{L^{\infty}}
		+ \norm{\eta, v} _{H_{\varkappa}^1 \times H^{1/2}}^2
		&
		\ \mbox{ for }
		s \in (1/2, 1)
		\\
		1 + \norm{\eta, v}
		_{H_{\varkappa}^{s + 1/4} \times H^{s - 1/4}}^2
		&
		\ \mbox{ for }
		s \geqslant 1
	\end{aligned}
	\right.
\]
which are the desired estimates.
\end{proof}

Knowing coercivity of the energy $E^s$, controlled
either by the smallness or
by the non-cavitation of the initial data,
one can deduce from the lemma that the time of existence
depends only on
\(
	\norm{\eta_0, v_0} _{H_{\varkappa}^{s' + 1/2} \times H^{s'}}
\),
where $1/2 < s' < s$.
Taking into account the boundedness of
\(
	\norm{\eta, v} _{H_{\varkappa}^1 \times H^{1/2}}
	,
\)
holding true at least for small initial data,
one can get a stronger result thanks to
the Brezis-Gallouet limiting embedding \eqref{Brezis_inequality}.
In order to exploit it we need the following
Gr\"onwall inequality.

\begin{lemma}
[Gr\"onwall inequality]
	Let $y$ be an absolutely continuous positive
	function defined on some interval $[0, T]$.
	Suppose that almost everywhere
	\[
		y' \leqslant Ay \log y
	\]
	where $A > 0$ is constant.
	Then there exists $C > 0$ independent of $T$
	such that
	\[
		y(t) \leqslant \exp \left( Ce^{At} \right)
		.
	\]
\end{lemma}

\begin{proof}

Denote the right hand side by $z(t) = \exp \left( Ce^{At} \right)$,
where we take $C > 0$ such that $z(0) > y(0)$.
Regard the derivative
\[
	\left(
		\frac yz
	\right)'
	=
	\frac{y'z - yz'}{z^2}
	\leqslant
	A
	\frac yz
	\log
	\frac yz
	,
\]
where the latter is less than zero
at least for $t = 0$.
So the fraction $y/z$ decreases and stays always below
the unity.

\end{proof}

\begin{corollary}
[Persistence of regularity]
\label{persistence_of_regularity_corollary}
	In the conditions of the a priori estimate lemma
	\ref{a_priori_estimate_lemma}
	the following holds true
	\[
		\norm{\eta(t), v(t)} _{H_{\varkappa}^{s + 1/2} \times H^s}
		\leqslant
		\exp
		\left(
			Ce^{C (1 + \varkappa) t}
		\right)
	\]
	provided $s < 1$,
	and if $s \geqslant 1$ then
	\[
		\norm{\eta(t), v(t)} _{H_{\varkappa}^{s + 1/2} \times H^s}
		\leqslant
		\norm{\eta_0, v_0} _{H_{\varkappa}^{s + 1/2} \times H^s}
		\exp
		\left(
			C (1 + \varkappa) t + C (1 + \varkappa) \int_0^t
			\norm{\eta, v}
			_{H_{\varkappa}^{s + 1/4} \times H^{s - 1/4}}^2
		\right)
	\]
	where the constant $C> 0 $ does not depend on $\varkappa$, $\mu$.
	In particular, the maximal time of existence
	$T^{\star} = +\infty$ provided
	\(
		\norm{\eta_0, v_0} _{H_{\varkappa}^1 \times H^{1/2}}
	\)
	is small enough.	
\end{corollary}

\begin{proof}
The statement is obvious for $s \geqslant 1$.
Suppose $s \in (1/2, 1)$.
By Lemma \ref{energy_bound_lemma}
and Corollary \ref{reg_energy_bound_corollary}
the norm
\(
	\norm{\eta(t), v(t)} _{H_{\varkappa}^1 \times H^{1/2}}
\)
stays bounded with time.
Hence from the Brezis-Gallouet inequality
\eqref{Brezis_inequality} one deduces
\[
	\lVert v(t) \rVert _{L^{\infty}}
	\lesssim 1 +
	\log \left( 3 + \lVert v(t) \rVert _{H^s} \right)
	.
\]
Thus applying Lemma \ref{persistence_of_regularity_lemma}
and taking into account that $E^s$ is coercive
one obtains
\[
	\frac{d}{dt} E^s
	\lesssim
	(1 + \varkappa)
	\left(
		1 +
		\log \left( 3 + E^s \right)
	\right)
	E^s
	.
\]
As a result,
after the application of the previous lemma
with $y = 3 + E^s$,
we have the estimate
\[
	E^s
	\leqslant
	\exp
	\left(
		C e^{C (1 + \varkappa) t}
	\right)
	,
\]
which again due to coercivity of $E^s$ leads to the first inequality
of the corollary after renaming the constant.

\end{proof}
%
%
\section{Proof of Theorems \ref{capillarity_theorem_local} and \ref{capillarity_theorem}}
\label{Proof_of_Theorem}
\setcounter{equation}{0}
%
%
With the a priori estimate \eqref{a_priori_estimate} in hand
we can reapply the local existence Lemma
\ref{regularised_local_existence_lemma}
for the regularised problem \eqref{reg_capillarity_sys}
with $\mu \in (0, 1)$ and $p = 1$
in order to obtain solutions $u^{\mu} = (\eta^{\mu}, v^{\mu})$
on the time interval $[0, T_0]$ defined by
Lemma \ref{a_priori_estimate_lemma}.
Convergence of $u^{\mu}$ as $\mu \to 0$
follows from an adaptation of Lemma \ref{difference_energy_lemma}
to the difference energy \eqref{difference_energy}
with $\eta_j = \eta^{\mu_j}$, $v_j = v^{\mu_j}$
($j = 1,2$) and $0 < \mu_2 < \mu_1 < 1$.
The proof repeats the arguments of
Lemma \ref{difference_energy_lemma} and
Lemma \ref{energy_dissipation_lemma}.
Moreover, using the Gagliardo--Nirenberg interpolation
one can obtain that $u^{\mu}$ converges to some $u$
in
\(
	C \left(
		[0, T_0]; H^{r + 1/2} (\mathbb R) \times H^r (\mathbb R)
	\right)
\)
as $\mu \to 0$ for any $0 < r < s$.
This $u$ is a solution of \eqref{capillarity_sys}
in the distributional sense.
Furthermore, to prove persistence
\(
	u \in
	C \left(
		[0, T_0]; H^{s + 1/2} (\mathbb R) \times H^s (\mathbb R)
	\right)
	,
\)
justify all the previous steps and obtain continuity
of the flow map one has to regularise the initial data
\eqref{capillarity_data} as
\(
	u_0^{\epsilon} = (
		\eta_0 * \rho_{\epsilon},
		v_0 * \rho_{\epsilon}
	)
	,
\)
where $\rho_{\epsilon}$ is an approximation of the identity
parametrised by $0 < \epsilon < 1$
\cite{Bona_Smith, Kato_Ponce87}.
An application of the Bona--Smith argument
in a straightforward standard way
\cite{Bona_Smith, Kenig_Pilod, Linares_Ponce}
results in the persistence and continuous dependence.
We omit further details.

%
%
\section{The two-dimensional problem}
\label{Two_dimensional_problem}
\setcounter{equation}{0}
%
%
In this section we comment briefly on adaptation of the
proof for the two dimensional case.
Firstly, we define the energy norm
\begin{equation}
\label{capillarity_norm2d}
	\norm{ \eta, \mathbf v }
	_{ H_{\varkappa}^{s + 1/2} \times H^s \times H^s }^2
	= \varkappa \lVert \nabla \eta \rVert _{H^{s - 1/2}}^2
	+ \lVert \eta \rVert _{H^{s - 1/2}}^2
	+ \lVert K^{-1} \mathbf v \rVert
	_{ H^{s - 1/2} \times H^{s - 1/2} }^2
\end{equation}
and the modified energy
\begin{equation}
\label{energy2d}
	E^s(\eta, \mathbf v)
	= \frac 12
	\norm{ \eta, \mathbf v }
	_{ H_{\varkappa}^{s + 1/2} \times H^s \times H^s }^2
	+ \frac 12 \int \eta \left| J^{s - 1/2 } \mathbf v \right| ^2
	,
\end{equation}
and then notice that it is coercive provided
the wave $\eta$ either satisfies the noncavitation condition
or has small $H^1$-norm.
Note that the latter does not imply the first one,
since now we do not have embedding of $H^1$
to $L^{\infty}$.
The smallness of
$H_{\varkappa}^1 \times H^{1/2} \times H^{1/2}$-norm
can be controlled by the energy conservation.
Indeed, by
H\"older's inequality and the Sobolev embedding
the cubic part of Hamiltonian \eqref{Hamiltonian2}
is estimated as
\[
	\int  \eta | \mathbf v |^2 dx
	\lesssim
	\norm{\eta} _{L^2}
	\norm{\mathbf v} _{H^{1/2} \times H^{1/2} }^2
	,
\]
and so repeating the arguments given in the proof
of Lemma \ref{energy_bound_lemma} we arrive at the conclusion
that the small enough initial data stays small through
the flow.
For $s > 2$ the noncavitation preserves locally-in-time
due to the first equation in \eqref{wt2d}.

The assumption
$\nabla \times  \mathbf v_0 = 0$
is needed to correctly define
the semigroup associated with the regularised
linear problem.
Indeed, instead of Semigroup \eqref{semigroup}, for
the two dimensional problem we have
\begin{equation*}
	\mathcal S(t) =
	\exp(-\varkappa \mu t |D|^p)
	\mathcal K
	\begin{pmatrix}
		\exp(-itK_{\varkappa}|D|) & 0
		\\
		0 & \exp(itK_{\varkappa}|D|)
	\end{pmatrix}
	\mathcal K^{-1}
	,
\end{equation*}
where
\[
	\mathcal K
	=
	\frac 1{\sqrt 2}
	\begin{pmatrix}
		1 & 1
		\\
		K_{\varkappa} \frac{D_1}{|D|} & -K_{\varkappa} \frac{D_1}{|D|}
		\\
		K_{\varkappa} \frac{D_2}{|D|} & -K_{\varkappa} \frac{D_2}{|D|}
	\end{pmatrix}
	, \quad
	\mathcal K^{-1}
	=
	\frac 1{\sqrt 2}
	\begin{pmatrix}
		1 & \frac{ |D| }{ 2 K_{\varkappa} D_1 }
		& \frac{ |D| }{ 2 K_{\varkappa} D_2 }
		\\
		1 & - \frac{ |D| }{ 2 K_{\varkappa} D_1 }
		& - \frac{ |D| }{ 2 K_{\varkappa} D_2 }
	\end{pmatrix}
\]
with $K_{\varkappa}$
defined by \eqref{K_kappa}.
Note that $\mathcal K^{-1}$ is well defined on
the subspace of 
\(
	H^{s + 1/2} \left( \mathbb R^2 \right) \times
	\left( H^s \left( \mathbb R^2 \right) \right)^2
\)
with the curl free second coordinate.
Moreover, it is easy to show that the condition
$\nabla \times  \mathbf v = 0$
preserves through the flow.
The energy estimates and the rest of the proof of
Theorem \ref{capillarity_theorem2d}
can be done in exactly the same manner as in the one dimensional
case, and so we omit further details.
%
%
%
%
\section{The low capillarity regime}
\label{Low_capillarity_regime}
\setcounter{equation}{0}
%
%
%
%
This section is devoted to analysis of the solution dependence 
on the surface tension $\varkappa \in (0, 1]$.
It allows, for instance, to validate that solutions of
Systems \eqref{capillarity_sys}, \eqref{wt2d}
with $\varkappa = 0$,
that are known to exist \cite{Dinvay_Tesfahun},
do indeed approximate solutions of the
same systems when $\varkappa \ll 1$.
We restrict ourselves to the one dimensional case.
The extension to the two dimensional situation is straightforward.

\begin{theorem}
	Let $s \geqslant 2$ and
	\[
		u^{\varkappa} = ( \eta^{\varkappa}, v^{\varkappa} )
		\in C \left( [0, T];
		H^{s + 1/2}(\mathbb{R}) \times H^s(\mathbb{R}) \right)
		\cap C^1 \left( (0, T);
		H^{s - 1}(\mathbb{R}) \times H^{s - 3/2}(\mathbb{R}) \right)
	\]
	be the solution of
	Problem \eqref{capillarity_sys}, \eqref{capillarity_data}
	for each $\varkappa \in (0, 1]$.
	Then $u^{\varkappa}$ converges to the solution
	$u = (\eta, v)$ of
	Problem \eqref{capillarity_sys}, \eqref{capillarity_data}
	with $\varkappa = 0$ in
	\(
		C \left( [0, T];
		H^{s - 1/2}(\mathbb{R}) \times H^s(\mathbb{R}) \right)
	\)
	as  $\varkappa \to 0$.
\end{theorem}
\begin{proof}
By the Bona-Smith argument it is enough to prove the statement
for the smooth initial data $u_0 = (\eta_0, v_0)$ with
\(
	\eta_0, v_0 \in
	H^{\infty}(\mathbb{R})
\).
Moreover, it is enough to prove convergence in
\(
	C \left( [0, T];
	L^2(\mathbb{R}) \times H^{1/2}(\mathbb{R}) \right)
\).
Without loss of generality we can assume that $T$
coincides with $T_0$ defined in Lemma \ref{a_priori_estimate_lemma}.
Note that it can be regarded as independent of $\varkappa \in (0, 1]$
according to Remark \ref{uniform_a_priori_estimate_remark}.
Moreover, we can assume that on the same time interval
$[0, T]$ the solution $u$, corresponding to the zero surface tension,
also satisfies \eqref{a_priori_estimate}
with the same constant $C$ and $\varkappa = 0$.

Define functions $\theta = \eta^{\varkappa} - \eta$,
$w = v^{\varkappa} - v$.
Then $\theta$ and $w$ satisfy the following system
%
%
\begin{equation*}
\left\{
\begin{aligned}
	\theta_t &=
	- \partial_x w - i \tanh D (\theta v + \eta^{\varkappa} w)
	, \\
	w_t &=
	- i \tanh D \theta
	- i \varkappa D^2 \tanh D \eta^{\varkappa}
	- i \tanh D ( ( v^{\varkappa} + v ) w ) / 2
	.
\end{aligned}
\right.
\end{equation*}
%
%
Introduce the norm
\begin{equation*}
	\norm{ \theta, w }^2
	=
	\norm{ \theta, w } _{ H_0^1 \times H^{1/2} }^2
	= \lVert \theta \rVert _{L^2}^2
	+ \lVert K^{-1} w \rVert _{L^2}^2
\end{equation*}
and calculate its derivative
\begin{multline*}
	\frac 12 \frac d{dt}
	\norm{ \theta, w }^2
	=
	-i \int \theta \tanh D \left( \theta v + \eta^{\varkappa} w \right)
	-i \varkappa \int \left( K^{-1} w \right)
	K^{-1} D^2 \tanh D \eta^{\varkappa}
	\\
	- \frac i2 \int \left( K^{-1} w \right)
	K^{-1} \tanh D ( ( v^{\varkappa} + v ) w )
	\lesssim
	\norm{ \theta } _{L^2}^2
	\norm{ v } _{H^1}
	+
	\norm{ \theta } _{L^2}
	\norm{ K^{-1} w } _{L^2}
	\norm{ \eta^{\varkappa} } _{H^{1/2}}
	\\
	+
	\varkappa
	\norm{ K^{-1} w } _{L^2}
	\norm{ \partial_x \eta^{\varkappa} } _{H^{3/2}}
	+
	\norm{ K^{-1} w } _{L^2}^2
	\left(
		\norm{ v^{\varkappa} } _{H^1}
		+
		\norm{ v } _{H^1}
	\right)
	.
\end{multline*}
Thus we have
\[
	\frac d{dt}
	\norm{ \theta, w }
	\lesssim
	\norm{ \theta, w }
	\left(
		\norm{ \eta^{\varkappa}, v^{\varkappa} }
		_{H_{\varkappa}^{3/2} \times H^1}
		+
		\norm{ v } _{H^1}
	\right)
	+
	\sqrt{\varkappa}
	\norm{ \eta^{\varkappa}, v^{\varkappa} }
	_{H_{\varkappa}^{5/2} \times H^2}
	,
\]
and so applying the a priori estimate
\eqref{a_priori_estimate} one deduces
\[
	\frac d{dt}
	\norm{ \theta, w }
	\lesssim
	\norm{ u_0 }
	_{H_{\varkappa}^{5/2} \times H^2}
	\left(
		\norm{ \theta, w }
		+
		\sqrt{\varkappa}
	\right)
	.
\]
Taking into account that
at the initial time moment
\(
	\theta(0) = w(0) = 0
\),
one easily obtains
\[
	\norm{ \theta(t), w(t) }
	\leqslant
	\sqrt{\varkappa}
	C \norm{ u_0 } _{H_{\varkappa}^{5/2} \times H^2}
	\left(
		\exp
		\left(
			C \norm{ u_0 } _{H_{\varkappa}^{5/2} \times H^2} t
		\right)
		- 1
	\right)
	,
\]
that tends to zero as $\varkappa \to 0$
uniformly with respect to $t \in [0, T]$.
This concludes the proof.
\end{proof}

\vskip 0.05in
\noindent
{\bf Acknowledgments.}
{
The author is grateful to Didier Pilod,
Achenef Tesfahun and Henrik Kalisch for
fruitful discussions and numerous helpful comments.
The research is supported by the
Norwegian Research Council.
}

\bibliographystyle{amsplain}
\bibliography{bibliography}

\end{document}